\documentclass[leqno]{siamltex704}
\usepackage{amsmath}
\usepackage{graphicx}
\usepackage{mathrsfs}
\usepackage{rotating}
\usepackage{lscape}
\usepackage{float}
\usepackage{amsfonts,amssymb}
\usepackage{dsfont}
\usepackage{pifont}
\usepackage{tikz}
\usepackage{wrapfig} % Allows in-line images
\usepackage{hyperref}
\usepackage{multirow}
\usepackage{bm}
\usepackage{mathtools}
\usepackage{subfigure}
\usepackage{colortbl}

\definecolor{mygray}{gray}{.9}
\definecolor{mypink}{rgb}{.99,.91,.95}
\definecolor{mycyan}{cmyk}{.3,0,0,0}
\def\argmin{\operatornamewithlimits{arg\ min}}

\def\diam{\operatornamewithlimits{diam}}

\def\T{{\mathcal T}}
\def\D{{\mathcal D}}
\def\E{{\mathcal E}}

\def\argmin{\arg\min}
\def\d{{\rm div}}
\def\bn{{\bf n}}

\def\bv{{\bf v}}
\def\pT{{\partial T}}
\def\pD{{\partial D}}
\def\3bar{{|\!|\!|}}

\newtheorem{algorithm}{Algorithm}[section]
\def\argmin{\operatornamewithlimits{arg\ min}}

\title{A High order Conservative Flux Optimization Finite Element Method for Diffusion Equations}

\author{Yujie Liu\thanks{Center for Quantum Computing, Peng Cheng Laboratory, Shenzhen 518005, China; Center for Mathematical Sciences, Huazhong University of Science \& Technology, Wuhan, China; (liuyj02@pcl.ac.cn). The research of Liu was partially supported by Guangdong Provincial Natural Science Foundation (No. 2017A030310285), Shandong Provincial natural Science Foundation (No. ZR2016AB15) and Youthful Teacher Foster Plan Of Sun Yat-Sen University (No. 171gpy118),}
\and Yue Feng \thanks{Department of Mathematics, Jilin University, Changchun,
China (yuefeng19@mails.jlu.edu.cn). The research of this author was supported in
part by China Natural National Science Foundation (91630201, U1530116, 11726102, 11771179, 93K172018Z01, 11701210, JJKH20180113KJ, 20190103029JH), and by the Program for Cheung Kong Scholars of Ministry of Education of China, Key Laboratory of Symbolic Computation and Knowledge Engineering of Ministry of Education, Jilin University, Changchun, 130012, P.R. China,}
\and Ran Zhang \thanks{Department of Mathematics, Jilin University, Changchun,
China (zhangran@jlu.edu.cn).}
}

\begin{document}

\maketitle

\begin{abstract}
This article presents a high order conservative flux optimization (CFO) finite element method for the elliptic diffusion equations.
The numerical scheme is based on the classical Galerkin finite element method enhanced by a flux approximation on the boundary of a prescribed set of arbitrary control volumes (either the finite element partition itself or dual voronoi mesh, etc).
The numerical approximations can be characterized as the solution of a constrained-minimization problem with constraints given by the flux conservation equations on each control volume.
The discrete linear system is a typical saddle-point problem, but with less number of degrees of freedom than the standard mixed finite element method, particularly for elements of high order.
Moreover, the numerical solution of the proposed scheme is of super-closeness with the finite element solution.
Error estimates of optimal order are established for the numerical flux as well as the primary variable approximations.
We present several numerical studies in order to verify convergence of the CFO schemes.
A simplified two-phase flow in highly heterogeneous porous media model problem will also be presented.
The numerical results show obvious advantages of applying high order CFO schemes.
\end{abstract}



\pagestyle{myheadings}

%
%\noindent {\bf Mathematics Subject Classification (2010)} 65N15;
%65N30; 41A30

\section{Introduction}
This paper is concerned with the development of high order numerical methods for partial differential equations that maintain important conservation properties for the underlying physical variables. For simplicity, consider the elliptic diffusion equation that seeks an unknown function $u=u(x)$ satisfying
\begin{equation}\label{EQ:Elliptic}
\left \{\begin{split}
-\nabla\cdot(\alpha\nabla u)&=f,\qquad {\rm in}\  \Omega\\
u&=g, \qquad {\rm on}\ \partial\Omega.
\end{split}\right.
\end{equation}
where $\Omega\subset {\bf R}^d (d=2,3)$ is a bounded polygonal ($d=2$) or polyhedral ($d=3$) domain with boundary $\partial\Omega$, and $\alpha=\{\alpha_{i,j}\}_{d\times d}$ is a symmetric, positive definite tensor; i.e., there exists a positive constant $\alpha_0$ such that
$$
\xi^T\alpha\xi\ge \alpha_0 \xi^T\xi, \qquad \forall  \; \xi \in \Omega.
$$

Let $\bv = - \alpha \nabla u$, the first equation of system (\ref{EQ:Elliptic}) can be rewritten as :
\begin{eqnarray}
\nabla\cdot\bv &=& f. \label{EQ:system-linear-eq2}
\end{eqnarray}
Consider an arbitrary control element $D$ in $\Omega$. By the divergence theorem, the equation (\ref{EQ:system-linear-eq2}) gives $\int_{\partial D} \bv\cdot\bn ds = \int_D f dx,$
where $\bn$ is the outward normal vector of $\partial D$. Define the flux variable $q =\bv\cdot\bn_{e}$ on the boundary of $D$, where $\bn_{e}$ is the normal vector of $\partial D$ with prescribed direction.
We have
\begin{equation}\label{EQ:conservation-eq2}
 \int_{\partial D} q  \bn\cdot\bn_{\partial D}ds = \int_D f dx.
\end{equation}
Equation \eqref{EQ:conservation-eq2} is the local mass conservation property, which is of key importance for many applications.
The readers may refer to \cite{LiuWang_SINUM_2017} for a detailed introduction.

The classical finite element method (FEM) seeks $u_h\in S_h$ such that
\begin{equation}\label{EQ:FEarg}
u_h=\argmin_{v\in S_h} J(v), \; J(v)=\frac{1}{2}a(v,v)-(f,v),
\end{equation}
where $S_h$ consists of $C^0$ piecewise polynomials of degree $k\geq 1$ on prescribed finite element partition $\T_h=\{T\}$ (Fig. \ref{fig:mesh}).
By taking the Fr\'echet derivative, the FEM approximation is given as the solution of :
\begin{equation}\label{EQ:FE}
(\alpha\nabla u_h,\nabla v)=(f,v),\; \forall v\in S_h.
\end{equation}
The FEM is easy to implement and is applicable to general domains.
The theory is also well-established.
However, since $\alpha\nabla u_h$ is usually not continuous from element to element,
the FEM can not give a locally conserved approximation of the flux variable directly.

In the literature, one can find various numerical methods designed for \eqref{EQ:Elliptic} that preserve the mass conservation property locally on each element $T\in\T_h$. One of such methods is the finite volume method (FVM) widely used in scientific computing for problems in science and engineering, including fluid dynamics \cite{Barth.T;Ohlberger2004,Emonot1992,EymardGallouetHerbin2000,LeVeque2002,Li.R2000,
Nicolaides1995,Shu2003}.
Most algorithms in FVM enjoy the nice feature of algorithmic simplicity and computational efficiency, and some of the low order FVMs (e.g., $P_0$ and $P_1$ schemes) have been well studied for their mathematical convergence and stability \cite{Bank.R;Rose.D1987,Cai.Z1991,ChouLi2000,Hackbusch.W1989a,Li.R2000,Suli_MCOM_1992,Yubo2012}. It should also be noted that the high order and symmetric FVMs are generally challenging in theory and algorithmic design \cite{CaiDouglasPark2003,Chen2010,ChenWuXu2012,LinYangZou2015}. In the finite element context, several conservative numerical schemes have been developed. The mixed finite element method \cite{RaviartThomas, bdm,Carstensen_MCAMS_1997}, the discontinuous Galerkin finite element method \cite{Arnold_SIAMJNA_2002}, the hybridizable discontinuous Galerkin \cite{Nguyen_JCP_2009}, and weak Galerkin finite element methods \cite{WangYe_2013, wy3655, WangWang_2016} are a few of such examples that give numerical approximations with conservative numerical flux. A first order conservative flux optimization (CFO) method \cite{LiuWang_SINUM_2017} has been proposed recently via a conservation-constrained optimization approach by using the continuous finite element space.
The scheme has been employed in the simulation of a simplified two-phase flow problem in highly heterogeneous porous media, which show the effectiveness and robustness of the scheme. In this paper, we propose a high order CFO scheme by modifying the continuous finite element space. This new scheme is locally conservative and can provide conservative high order flux on arbitrary control volumes.
Moreover, it has less degree of freedoms than mixed finite elements methods (especially for high order elements).
We established optimal order error estimates for both the primal and the flux variables.
Numerical simulation of the two-phase flow problem in highly heterogeneous porous media show obvious advantages of applying the high order CFO schemes.

The paper is organized as follows. In Section \ref{sectionccfv}, we present the high order conservative flux optimization finite element scheme for the model problem \eqref{EQ:Elliptic}. In Section \ref{sectionWpS}, we establish a result on the well-posedness and stability for the high order conservative flux optimization scheme. In Section \ref{sectionEE}, we derive some error estimates for the resulting numerical approximations in various Sobolev norms. Finally, in Section \ref{numerical-experiments}, we present several numerical results to demonstrate the efficiency and accuracy of the new scheme.

\section{A High Order Conservative Flux Optimization Scheme}\label{sectionccfv}
Let $\T_h$ be a regular finite element triangular partition of the polygonal domain $\Omega\subset\mathbb{R}^2$ (Fig. \ref{fig:mesh}--solidline). Assume that we have a set of control volumes $\D_h\coloneqq\{D\}$ as illustrated in (Fig. \ref{fig:mesh}--blue domain), with $\partial \D_h$ the set of edges of $D \in \D_h$. We note that $\D_h$ is arbitrary and can be completed independent of the finite element partition $\T_h$. In practical computation, the set of control volumes could be taken either as the dual Voronoi mesh or the finite element partition $\T_h$ itself. Denote by $h\coloneqq \max_{{D} \in \D_h}h_{D}$ the meshsize of $\D_h$, where $h_{D}=\diam({D})$ is the diameter of the element ${D} \in \D_h$. Denote by $\E_h$ the edge set of $\D_h$, and $\E_h^0 \subset \E_h$ the set of all interior edges. The set of boundary edges is denoted as $\E_h^B\coloneqq\E_h \backslash\E_h^0$. The diameter of the edge $e\in \E_h$ is denoted as $h_e=\diam(e)$. For convenience, for each $e\in\E_h$ we assign a normal direction $\bn_e$ which provides an orientation for $\E_h$.
\begin{figure}
\centering
~~~~~~\includegraphics [width=0.7\textwidth]{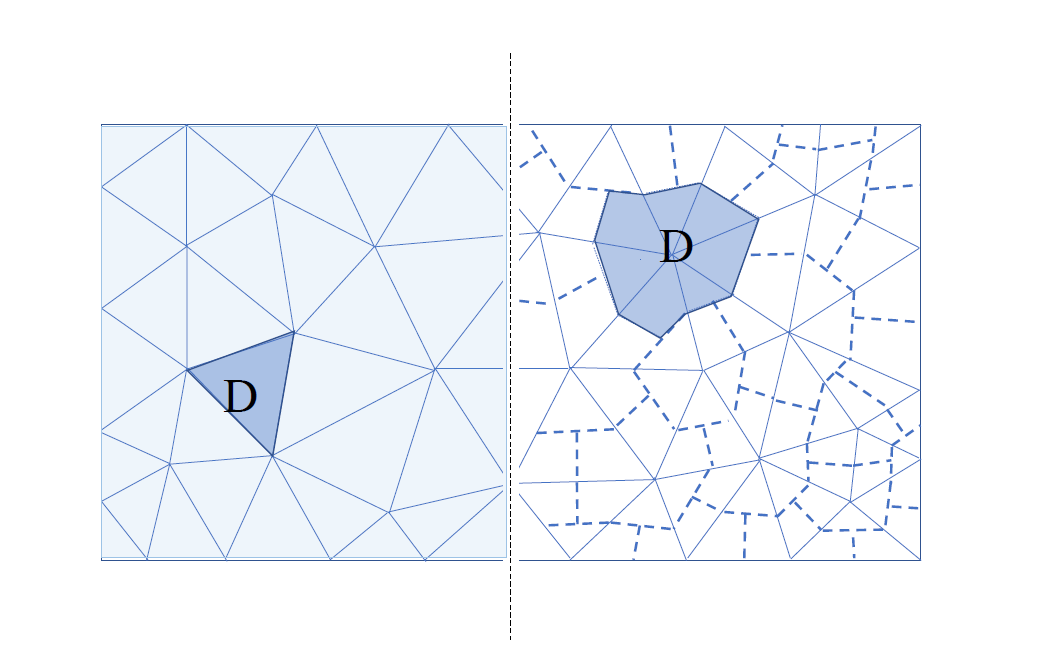}
\caption{Illustration of the control volumes which would be either the finite element triangular partition (left, solidline) or control volumes of the dual Voronoi mesh (right, dotted line) of domain $\Omega$.}
\label{fig:mesh}
\end{figure}

Let $S_h \subset H_0^1(\Omega)$ be the finite element space, that is the classical {\color{black} $C^0$-conforming} element given by
\[
S_h=\{v\in C^0(\Omega): \ v|_{T}\in {P}_k(T), \forall {T}\in\T_h, v|_{\partial\Omega}=0\}.
\]
Here $C^0(\Omega)$ stands for the space of continuous functions in the domain $\Omega$. Denote by $V_h$ the flux space consisting of  piecewise polynomial of degree $k-1$ on $\E_h$ and $W_h$ the space of piecewise constant functions on $\E_h$ given as follows:
\[
V_h=\{q \in C^{-1}(\E_h): \ \textcolor{black}{q|_{e}\in {P}_{k-1}(e)}, \forall {e}\in\E_h\},
\]
\[
W_h=\{\sigma:  \sigma|_{D}\in {P}_0(D), \forall {D}\in\D_h\}.
\]
See Figure \ref{fig:triangle} for an illustration of spaces $S_h$ and $V_h$ where the control volume $D$ coincides with the finite element triangular element $T$.
\begin{figure}
\begin{center}
\includegraphics [width=0.4\textwidth]{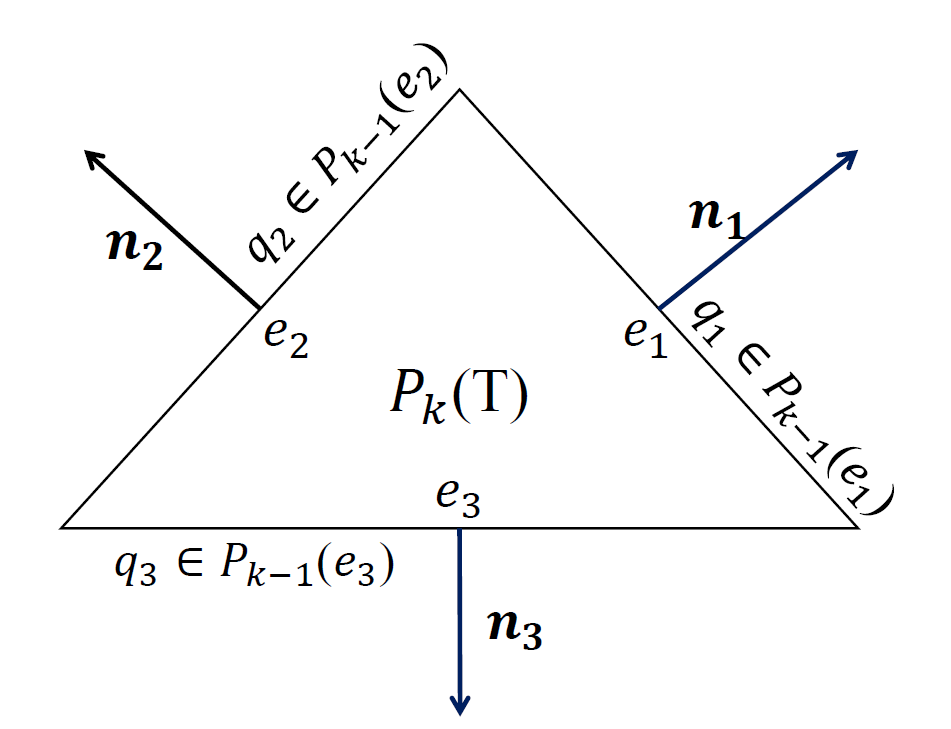}
\end{center}
\caption{An illustrative triangular element with local flux}
\label{fig:triangle}
\end{figure}
For any $q\in L^2(\E_h)$, denoted by $\nabla_w\cdot q$ the discrete weak divergence given as a function in $W_h$ such that on each $D\in\D_h$
\begin{equation}\label{EQ:001}
(\nabla_w\cdot q)|_D =\frac{1}{|D|}\int_{\partial D} q \bn\cdot {\bf n}_e ds.
\end{equation}
A flux function $p\in V_h$ is said to be {\it locally conservative} if it satisfies
\begin{equation}\label{EQ:conservation_fv}
(\nabla_w\cdot p, w)=(f, w),\qquad \forall w\in W_h.
\end{equation}

The classical Ritz-Galerkin finite element method seeks $R_h u\in S_h$ such that
\begin{eqnarray}\label{EQ:Ritz-Galerkin-1}
(\alpha\nabla R_h u, \nabla v) =(f,v), \forall v \in S_h.
\end{eqnarray}
A straightforward numerical flux would be given by $q_h^* = - \alpha \nabla (R_hu) \cdot \bn_e$ on each edge $e\in \E_h$, but this numerical flux function $q_h^*$ is usually discontinuous across each edge $e\in \E_h$ and/or has difficulty to be locally conservative on each control element $D$. A remedy to this challenge is to find a pair $(u_h;q_h)\in S_h\times V_h$ that satisfies the mass conservation equation (\ref{EQ:conservation_fv}) while the flux error $q_h + (\alpha \nabla u_h \cdot \bn_e)$ is minimized in a metric at the user's discretion. To this end, we introduce a functional in the space $S_h\times V_h$ as follows:
\begin{equation}\label{EQ:functional}
J_{r,\beta}(v,p):= \frac{1}{r}\sum_{D\in\D_h}h_D^\beta \sum_{e\subset\partial D}  \int_e |p+\alpha^*\nabla v \cdot \bn_e|^r ds {\color{black} + \frac12 (\alpha \nabla v, \nabla v) - (f,v)},
\end{equation}
where $r\in [1,\infty)$ is a prescribed value, $\beta$ is a parameter with real value. We note that $\alpha^*$ is the trace of $\alpha$ from  inside of $D$ to $\pD$ in case $\alpha$ is discontinuous at $\pD$, we write $\alpha$ instead of $\alpha^*$ for simplicity in the following. Our numerical algorithm then seeks $(u_h;q_h)\in S_h\times V_h$ which minimizes the functional $J_r$ under the constraint (\ref{EQ:conservation_fv}).
\begin{algorithm}
Find $u_h\in S_h$ and $q_h\in V_h$ such that
\begin{eqnarray}\label{EQ:argmin}
&&(u_h;q_h)=\argmin_{v\in S_h, p\in V_h, (\nabla_w\cdot p, w)=(f, w), \forall w\in W_h} J_{r,\beta}(v,p).
\end{eqnarray}
\end{algorithm}

Our new method essentially looks for a conservative flux variable that best approximates the obvious, but non-conservative numerical velocity $q_h = - \alpha\nabla u_h \cdot \bn_e $ in a discrete metric. Following \cite{LiuWang_SINUM_2017}, the scheme is also named as {\em Conservative Flux Optimization (CFO)} finite element method in this article. It should be pointed out that the {\em CFO} finite element method was originally motivated by the idea of the primal-dual weak Galerkin method (namely, PDE-constraint minimization of stabilizers) presented as in \cite{WangWang_2016} for the second order elliptic equation in non-divergence form.

%Due to the emphasis on the mass conservation and the error minimization for the flux approximation, the numerical scheme (\ref{EQ:argmin}) is named the {\em Conservative Flux Optimization Finite Element Method} or the {\em CFO-FEM} in brief.

The Euler-lagrange formulation for the scheme \eqref{EQ:argmin} seeks $(u_h;q_h;\lambda_h)\in S_h\times V_h\times M_h$ such that
{\color{black}
\begin{eqnarray}\label{min-LagrangeForm-1}
 \langle D J_{r,\beta}(u_h;q_h), (v;p)\rangle + (\nabla_w\cdot p, \lambda_h)&=& 0,\qquad \forall\; (v;p)\in S_h\times V_h,\\
(\nabla_w\cdot q_h, w) & = & (f,w),\qquad \forall\; w\in W_h, \label{min-LagrangeForm-2}
\end{eqnarray}
where
\begin{equation}\label{EQ:FDerivative-1}
\begin{split}
\langle D J_{r,\beta}(u_h;q_h), (v;p)\rangle :=& \sum_{T\in\D_h} h_D^\beta \sum_{e\subset\pD} \int_e {\color{black} (q_h+\alpha\nabla u_h \cdot \bn_e)( p+ \alpha\nabla v \cdot \bn_e)} ds\\
& \ + (\alpha \nabla u_h, \nabla v) - (f,v),\\
=& \sum_{T\in\D_h} h_D^\beta \sum_{e\subset\pD} \langle q_h+\alpha\nabla u_h \cdot \bn_e, p+ \alpha \nabla v \cdot \bn_e \rangle_e\\
 & \ + (\alpha \nabla u_h, \nabla v) - (f,v)
 \end{split}
\end{equation} }
is the Fr\'echet derivative of the functional $J_{r,\beta}$ at $(u_h;q_h)$ along the direction of $(v;p)$.
By introducing the following bilinear form
\begin{equation}\label{EQ:FDerivative-2}
s_h((u_h;q_h), (v;p)):=  (\alpha\nabla u_h, \nabla v) + \sum_{D\in\D_h} \sum_{e\subset\pD} h_D^\beta \langle q_h+\alpha\nabla u_h \cdot \bn_e, p+\alpha\nabla v \cdot \bn_e \rangle_e,
\end{equation}
we can rewrite the Euler-Lagrange equations (\ref{min-LagrangeForm-1})-(\ref{min-LagrangeForm-2}) as follows
\begin{eqnarray}
s_h((u_h;q_h), (v;p)) + (\nabla_w\cdot p, \lambda_h)&=& (f,v),\qquad \forall (v;p)\in S_h\times V_h,\label{min-LagrangeForm2-1}\\
(\nabla_w\cdot q_h, w) & = & (f,w),\qquad \forall w\in W_h. \label{min-LagrangeForm2-2}
\end{eqnarray}

The number of degrees of freedom (dof) for the scheme (\ref{min-LagrangeForm2-1})-(\ref{min-LagrangeForm2-2}) on a triangular element $T$ is:
\begin{eqnarray}
 N_{Dof} &=&N_{Dof}(P_{k} (T)) + \sum_{e\in\pT} N_{Dof}(P_{k-1} (e)) + N_{Dof}(P_{0} (T)) \\
         &=&\frac{(k+1)(k+2)}{2} + 3k +1, \nonumber
\end{eqnarray}
which is smaller than that of the Raviart-Thomas element $(k+1)(k+3)$ \cite{RaviartThomas_MFE_1977} and the Brezzi-Douglas-Marini (BDM) element $(k+1)(k+2)$ \cite{Carstensen_MCAMS_1997} for $k\geq3$, see Table \ref{tableDOF} for detail.

\begin{table}[h]
\caption{Number of dof for different methods on a triangle element $T$}\label{tableDOF}
\begin{center}
\begin{tabular}{||c|c|c|c||}
\hline
Order &  CFO  &  RT & BDM \\
\hline
$1$ & $ 7 $ & $8$  & $6$ \\
\hline
$2$ & $13 $ & $15$ & $12$\\
\hline
$3$ & $20 $ & $24$ & $20$ \\
\hline
$\vdots$ & $\vdots$ & $\vdots$ &$ \vdots$ \\
\hline
$k$ & $\frac{(k+1)(k+2)}{2} + 3k +1$ & $(k+1)(k+3)$ &  $(k+1)(k+2)$\\
\hline
\end{tabular}
\end{center}
\end{table}

\section{Solution Existence and Uniqueness}\label{sectionWpS}
In this section, we shall study the solution existence and uniqueness for the CFO-FEM scheme (\ref{EQ:argmin}) with $r=2$. Note that the corresponding Euler-Lagrange formulation is a system of linear equations given by (\ref{min-LagrangeForm2-1})-(\ref{min-LagrangeForm2-2}).

\begin{theorem}
Assume that the coefficient $\alpha$ is in $L^\infty(\D_h)$ and uniformly positive definite. Then the numerical scheme (\ref{min-LagrangeForm2-1})-(\ref{min-LagrangeForm2-2}) has one and only one solution $(u_h;q_h;\lambda_h)\in S_h \times V_h\times W_h$.
\end{theorem}

\begin{proof}
It suffices to show that the homogeneous problem has only trivial solution. To this end, let $(u_h;q_h;\lambda_h)\in S_h \times V_h\times W_h $ be a solution of the scheme (\ref{min-LagrangeForm2-1})-(\ref{min-LagrangeForm2-2}) with homogeneous data $f=0$.
By letting $v = u_h,\; p=q_h,\;w=\lambda_h$ in (\ref{min-LagrangeForm2-1})-(\ref{min-LagrangeForm2-2}), we have
\begin{equation}
(\alpha \nabla u_h, \nabla u_h) +
\sum_{D\in\D_h} \sum_{e\subset\pD} h_D^\beta \int_e|q_h+\alpha\nabla u_h\cdot\bm{n}_e|^2 ds=0,
\end{equation}
which leads to $u_h \equiv 0$ and $q_h\equiv 0$.

Next, from the equation (\ref{min-LagrangeForm2-1}) we have
\begin{equation}\label{EQ:April:002}
0=(\nabla_w\cdot p, \lambda_h) = \sum_{D\in\D_h}\lambda_h|_D\int_{\partial D} p \bn \cdot \bn_e ds, \quad \forall p\in V_h.
\end{equation}
Consider the following auxiliary problem: Find $\phi\in H_0^1(\Omega)$ such that
\begin{eqnarray*}
-\nabla\cdot(\alpha\nabla \phi )&=&\lambda_h,\qquad {\rm in}\  \Omega  \label{ellipticbdy-2}\\
\phi &=&0,\qquad {\rm on}\ \partial\Omega. \label{ellipticbc-2}
\end{eqnarray*}
Let $Q_b$ be the $L^2$ projection operator onto $V_h$, and set $p|_e=Q_b(- \alpha \nabla \phi \cdot \bn_e)$ on each edge $e\in \E_h$. It follows that
\begin{eqnarray*}
\int_{\pD} p \bn \cdot \bn_e ds &=& \int_{\pD} Q_b (-\alpha\nabla \phi \cdot \bn) ds \\
& = & \int_{\pD}-\alpha\nabla \phi \cdot \bn ds \\
& = & \int_{D} \lambda_h dx= \lambda_h|_D \cdot |D|.
\end{eqnarray*}
Substituting the above into (\ref{EQ:April:002}) yields
{\color{black}
\begin{equation}
\sum_{D\in\D_h}(\lambda_h|_D)^2 |D|=0,
\end{equation} }
which gives $\lambda_h\equiv 0$.
\end{proof}

\section{Error Estimates}\label{sectionEE}
The goal of this section is to establish some error estimate for the approximate solution $(u_h;q_h)$ arising from the scheme (\ref{min-LagrangeForm2-1})-(\ref{min-LagrangeForm2-2}). To this end,
denote by
\begin{equation}\label{equation.error}
e_u:= u_h -R_h u,\; e_q:=q_h -Q_b q,\; e_\lambda:=\lambda_h-\lambda,
\end{equation}
the error functions for the numerical solution, where
$q=-\alpha\nabla u \cdot \bm{n}_e$ is the flux on edge $e\in \E_h$, $Q_b q$ is the $L^2$ projection of the flux in $V_h$, and
$R_hu$ is the classical Ritz-Galerkin finite element approximation of $u$ given by (\ref{EQ:Ritz-Galerkin-1}). The convergence behavior for $R_hu$ has been extensively studied in existing literature in various Sobolev norms, and they are considered as standard and known in this paper.

\begin{theorem}\label{thm:Super}
For the model problem (\ref{EQ:Elliptic}), assume that the coefficient $\alpha$ is in $L^\infty(\D_h)$ and uniformly positive definite, $u\in H^{k+1}(\Omega)$. Let $(u_h;q_h)\in S_h\times V_h$ be the approximate solution arising from the numerical scheme (\ref{min-LagrangeForm2-1})-(\ref{min-LagrangeForm2-2}) with any real $\beta$, then the following error estimates hold true
\begin{eqnarray}\label{eq.primal}
\|u_h-R_h u\|_1\lesssim h^{k+\frac{\beta-1}{2}}\|u\|_{k+1},
\end{eqnarray}
\begin{equation}\label{error.Flux}
\left(\sum_D h_D \int_\pD| q_h- Q_bq|^2 ds\right)^{\frac12}
\lesssim h^{k}(1+ h^{(\beta-1)/2})\| u\|_{k+1}.
\end{equation}
In other words, for $\beta \ge 1$, the flux approximation is of optimal order and the numerical approximation for the primal variable is of super-closeness to the standard Galerkin finite element solution $R_h u$ in the usual $H^1$-norm.
\end{theorem}

\begin{proof}
 Note that $\lambda = 0$ in \eqref{equation.error}. It follows from (\ref{min-LagrangeForm2-1}) and \eqref{EQ:Ritz-Galerkin-1} that
\begin{equation} \label{EQ:April:001}
\begin{split}
&s_h((e_u;e_q), (v;p)) + (\nabla_w \cdot p, e_\lambda) \\
=& s_h((u_u;q_h), (v;p)) + (\nabla_w \cdot p, \lambda_h) - s_h((R_h u; Q_b q), (v;p)) \\
= & (f, v) - s_h((R_h u; Q_b q), (v;p)) \\
=&(f,v) - \sum_D h_D^\beta \sum_{e\in\pD}\langle Q_b q + \alpha \nabla R_h u \cdot \bm{n}_e, p+ \alpha \nabla v \cdot \bm{n}_e \rangle_{e}-(\alpha\nabla R_h u,\nabla v) \\
=& -\sum_D h_D^\beta \sum_{e\in\pD}\langle Q_b q + \alpha\nabla R_h u \cdot \bm{n}_e, p+\alpha\nabla v \cdot \bm{n}_e\rangle_e
\end{split}
\end{equation}
for all $p\in V_h$ and $v\in S_h$.
Next, from (\ref{min-LagrangeForm2-2}) and the definition of the $L^2$ projection operator $Q_b$ we have
\begin{equation*}%\label{EE-2}
\begin{split}
( \nabla_w \cdot  e_q, w ) =&( \nabla_w \cdot q_h, w ) - (\nabla_w \cdot Q_b q, w) \\
=& (f,w)-(\nabla_w \cdot q, w)
= 0
\end{split}
\end{equation*}
for all $w\in W_h$. In particular, we have
\begin{equation}\label{EE-2}
\begin{split}
( \nabla_w \cdot  e_q, e_\lambda ) =&( \nabla_w \cdot e_q, \lambda_h ) - (\nabla_w \cdot e_q, \lambda) \\
=& ( \nabla_w \cdot e_q, \lambda_h ) = 0.
\end{split}
\end{equation}
Thus, by choosing $p= e_q$ and $v= e_h$ in equation \eqref{EQ:April:001} we arrive at
\begin{equation} \label{EQ:April:005}
s_h((e_u;e_q), (e_u;e_q))= -\sum_D h_D^\beta \sum_{e\in\pD}\langle Q_b q + \alpha\nabla R_h u \cdot \bm{n}_e, e_q+\alpha\nabla e_u \cdot \bm{n}_e\rangle_e,
\end{equation}
which, combined with
$$
s_h((v;p), (v,p))= (\alpha \nabla v, \nabla v) +
\sum_D h_D^\beta \sum_{e \in \pD}\int_e |p+\alpha \nabla v \cdot \bm{n}_e|^2 ds ,
$$
leads to the following
\begin{equation} \label{EQ:April:006}
\begin{split}
& (\alpha \nabla e_u, \nabla e_u) + \sum_D h_D^\beta \sum_{e \in \pD}\int_e |e_q+\alpha \nabla e_u \cdot \bm{n}_e|^2 ds  \\
= &-\sum_D h_D^\beta \sum_{e\in\pD}\langle Q_b q + \alpha\nabla R_h u \cdot \bm{n}_e, e_q+\alpha\nabla e_u \cdot \bm{n}_e\rangle_e.
\end{split}
\end{equation}

Thus, by using the Cauchy-Schwarz inequality we have
\begin{equation}\label{EQ:April:008}
\begin{split}
& (\alpha \nabla  e_u,\nabla  e_u) + \sum_D h_D^\beta \sum_{e \in \pD} \int_e  |e_q + \alpha \nabla  e_u\cdot \bm{n}_e|^2 ds \\
\le & \ \sum_D h^\beta_D \sum_{e \in \pD} \| Q_b q + \alpha\nabla R_h u \cdot \bm{n}_e \|_e^2 \\
\lesssim & \ \sum_D h^\beta_D\sum_{e \in \pD} \left(\|Q_b q -q\|^2_e +\|\alpha \nabla (u-  R_h u)\cdot \bm{n}_e\|^2_e \right) \\
\lesssim & \ h^\beta \sum_D(h^{-1}\|\nabla (u- R_h u)\|^2_D + h\|\nabla^2 (u- R_h u)\|^2_D)\\
\lesssim & \ h^{\beta -1} h^{2k} \|u\|_{k+1}^2,
\end{split}
\end{equation}
which gives
\begin{eqnarray}\label{EE-q}
\sum_D h_D \sum_{e \in \pD}\| e_q + \alpha\nabla  e_u \cdot \bm{n}\|^2_{e} \lesssim  h^{2k} \|u\|_{k+1}^2
\end{eqnarray}
and
\begin{eqnarray}\label{EE-q2}
\|\nabla e_u \|^2_0 \lesssim  h^{\beta-1}h^{2k}\| u\|_{k+1}^2.
\end{eqnarray}
The last inequality confirms the error estimate \eqref{eq.primal}.

As to the flux error estimate \eqref{error.Flux}, we use the equation \eqref{EE-q} and \eqref{EE-q2} to obtain
\begin{equation}\label{eq.Flux}
\begin{split}
\sum_D h_D \int_\pD| q_h- Q_bq|^2 ds = & \sum_D h_D \int_\pD| (e_q + \alpha \nabla e_u\cdot\bn_e) - \alpha \nabla e_u \cdot\bn_e|^2 ds\\
\lesssim & \sum_D h_D \int_\pD \left(|e_q + \alpha \nabla e_u\cdot\bn_e|^2 + |\alpha \nabla e_u \cdot\bn_e|^2\right) ds\\
\lesssim & (\alpha \nabla e_u, \nabla e_u) + \sum_D h_D \int_\pD |e_q + \alpha \nabla e_u\cdot\bn_e|^2 ds\\
\lesssim & h^{2k}(1+ h^{\beta-1})\| u\|_{k+1}^2,
\end{split}
\end{equation}
which leads to \eqref{error.Flux} by taking the square-root of both sides. This completes the proof of the theorem.
\end{proof}

\medskip

An $L^2$ error estimate can be obtained by using the standard Poincar\'e inequality as follows:
\begin{equation}\label{Error:L2}
\begin{split}
\|e_u\|_0 \leq& C\|\nabla e_u\|_0\\
          \leq& C h^{k +\frac{\beta-1}{2}}\|u\|_{k+1}.
          \end{split}
\end{equation}
Thus, we have the optimal order error estimate in $L^2$ when $\beta \geq 3$. This result will be further confirmed or enhanced through numerical experiments to be presented in the following section.

\section{Numerical Experiments}\label{numerical-experiments}
In this section, we shall numerically verify the theoretical error estimates developed in the previous section. Our test problems are defined in two-dimensional square domains which seeks $u\in H^1(\Omega)$ such that
\begin{equation}\label{EQ:EllipticTest}
\left \{\begin{split}
-\nabla\cdot(\alpha\nabla u)&=f,\qquad {\rm in}\  \Omega\\
u&=g, \qquad {\rm on}\ \partial\Omega.
\end{split}\right.
\end{equation}
For each of the test cases, the CFO-FEM scheme \eqref{min-LagrangeForm2-1}-\eqref{min-LagrangeForm2-2} is implemented with $k=1,2,3$ and $r=2$. For simplicity, the control volumes are chosen as the finite element triangular partition $\T_h$ in all the numerical tests. The following metrics are used to measure the magnitude of the error:
\begin{eqnarray*}
\text{$L^2$-norm: }
&&\|u_h-u \|_{0}=\left(\sum_{{T}\in\T_h }\int_{{T}}|u_h-u|^2 d{T}\right)^{1/2},\\
\text{$H^1$-norm: }
&&\|u_h-u\|_1=\|\nabla(u_h-u) \|_{0},\\
\text{Flux $L^2$-norm: }
&& \|q_h -q \|=\left(\sum_{{T}\in\T_h }h_T\int_{\partial {T}}|\alpha \nabla u \cdot\bn_{e}+q_h|^2 d s\right)^{1/2}.
\end{eqnarray*}
The error between the classical Ritz-Galerkin finite element solution and the CFO-FEM solution will also be computed in the usual $L^2$ and $H^1$ norms; i.e., $\|u_h-R_hu\|_{0}$ and $\|u_h-R_hu \|_{1}$.

\subsection{Test Case 1: Smooth Coefficients} In this experiment, the elliptic problem \eqref{EQ:EllipticTest} has exact solution $u=\cos(\pi x)\cos(\pi y)$ with domain $\Omega=(0,1)^2$. The coefficient $\alpha$ is the identity matrix. The right-hand side function $f$ and the Dirichlet boundary data $g$ are chosen to match the exact solution.  The meshes are obtained by first uniformly partitioning the square domain $\Omega$ into $N^2,\ N=\frac{1}{h}$, small squares and then decomposing each small square into $2$ similar triangles. Table \ref{table11-h12} illustrates the performance of the CFO-FEM scheme with $k=1, 2, 3$, $\beta=1,2$ and $r=2$ on the uniform partitions. The results clearly show that the errors $\|u_h-u \|_{1}$ and $\|q_h-q \|_{0}$ converge to zero at optimal order of $h^k$ $(k=1,2,3)$ when $\beta\ge 1$. In addition, $\|u_h-u \|_{0}$ converges to zero at the order of $h^k$ or better. The $L^2$ convergence can be improved to be of optimal order with large values of $\beta$. For example, for the second order CFO-FEM (i.e, $k=2$), the numerical results indicate an optimal order of convergence with $\beta=2$ and $3$, while for the first and third order scheme, the numerical results suggest optimal order of convergence with $\beta=1$. For the numerical Lagrange multiplier $\lambda_h$, it converges to zero with rate $h^{k+\beta}$ for the first and third order scheme, while rate $h^{k+\beta-1}$ is observed for the second order scheme. The numerical results (see Tab. \ref{table11-FE-beta1}) also confirm the super-closeness estimate between the Ritz-Galerkin finite element solution and the CFO-FEM solution. The numerical results are in great consistency with the theory.

We note that, for clarity, only some representative results are presented in this section. The readers may refer to the appendix for a collection of additional results. For example, the the numerical results for the CFO-FEM scheme with $k=1, 2, 3$, $\beta=-1,0$ are included. The numerical results converge with certain orders, although we haven't theory estimations for these cases. Moreover, we observe that the errors $\|u_h-u \|_{1}$ and $\|q_h-q \|_{0}$ converge to zero at nearly optimal orders of $h^k$ $(k=1,2)$ when $\beta=0, 1$.

It should be emphasized that the main advantage of the CFO-FEM algorithm (\ref{min-LagrangeForm2-1}-\ref{min-LagrangeForm2-2}) is that one obtains not only a discrete solution $u_h$ with optimal order of convergence, but also an element-wise conserving flux $q_h$ with optimal order of convergence. Moreover, the CFO-FEM algorithm makes use of less number of degrees of freedom than some other existing numerical methods such as the mixed finite element method.

\begin{table}[!h]
\scriptsize
\begin{center}
\caption{Error and convergence performance of the CFO scheme for Test Case 1 on uniform meshes.}\label{table11-h12}
\begin{tabular}{|c|cc|cc|cc|cc|}
\hline
%\rowcolor{mygray}
\multicolumn{9}{|>{\columncolor{yellow!20}}c|}{ First order CFO with $\beta =1$ }\\
%\multicolumn{3}{>{\columncolor{red}}l}{\color{white}\textsf{LONDON}}
\hline
h & $\|u_h-u \|_{0}$ & rate & $\|u_h-u\|_{1}$ & rate &$ \|q_h-q \|_{0}$& rate  & $\|\lambda_h\|_{0}$ & rate  \\
\hline
1/8   &   7.44e-02 &  -     &   1.02e-01  & \color{blue}{-   } &    8.46e-02 &  \color{blue}{-}    &    1.69e-01 &  -       \\
1/16  &   1.94e-02 &  1.94  &   4.96e-02  & \color{blue}{1.04} &    4.07e-02 &  \color{blue}{1.06 }&    4.39e-02 &  1.94    \\
1/32  &   4.91e-03 &  1.98  &   2.46e-02  & \color{blue}{1.01} &    2.01e-02 &  \color{blue}{1.02 }&    1.11e-02 &  1.98    \\
1/64  &   1.23e-03 &  2.00  &   1.23e-02  & \color{blue}{1.00} &    1.00e-02 &  \color{blue}{1.00 }&    2.78e-03 &  2.00    \\
1/128 &   3.08e-04 &  2.00  &   6.14e-03  & \color{blue}{1.00} &    5.01e-03 &  \color{blue}{1.00 }&    6.95e-04 &  2.00    \\
\hline
\multicolumn{9}{|>{\columncolor{yellow!20}}c|}{First order CFO with $\beta =2$ }\\
\hline
h & $\|u_h-u \|_{0}$ & rate & $\|u_h-u\|_{1}$ & rate &$ \|q_h-q \|_{0}$& rate  & $\|\lambda_h\|_{0}$ & rate   \\
\hline
1/8   &    5.79e-02 &  -     &   1.00e-01 & \color{blue}{ -    } &   8.27e-02  &\color{blue}{ -    } &   2.45e-02  & -         \\
1/16  &    1.32e-02 &  2.13  &   4.92e-02 & \color{blue}{ 1.02 } &   4.03e-02  &\color{blue}{ 1.04 } &   2.87e-03  & 3.09      \\
1/32  &    3.02e-03 &  2.13  &   2.46e-02 & \color{blue}{ 1.00 } &   2.01e-02  &\color{blue}{ 1.01 } &   3.36e-04  & 3.10      \\
1/64  &    7.07e-04 &  2.09  &   1.23e-02 & \color{blue}{ 1.00 } &   1.00e-02  &\color{blue}{ 1.00 } &   4.01e-05  & 3.07      \\
1/128 &    1.70e-04 &  2.06  &   6.14e-03 & \color{blue}{ 1.00 } &   5.01e-03  &\color{blue}{ 1.00 } &   4.88e-06  & 3.04      \\
\hline
\hline
\multicolumn{9}{|>{\columncolor{mypink}}c|}{Second order CFO with $\beta =1$ }\\
\hline
h & $\|u_h-u \|_{0}$ & rate & $\|u_h-u\|_{1}$ & rate &$ \|q_h-q \|_{0}$& rate  & $\|\lambda_h\|_{0}$ & rate  \\
\hline
1/8   &    1.70e-02 &  -     &   1.15e-02  & \color{blue}{-   }  &   7.12e-02 & \color{blue}{  -   }  &   7.66e-03 &  -       \\
1/16  &    4.24e-03 &  2.00  &   2.88e-03  & \color{blue}{2.00}  &   1.73e-02 & \color{blue}{  2.04}  &   2.07e-03 &  1.89    \\
1/32  &    1.06e-03 &  2.00  &   7.19e-04  & \color{blue}{2.00}  &   4.27e-03 & \color{blue}{  2.02}  &   5.26e-04 &  1.97    \\
1/64  &    2.65e-04 &  2.00  &   1.80e-04  & \color{blue}{2.00}  &   1.06e-03 & \color{blue}{  2.01}  &   1.32e-04 &  1.99    \\
1/128 &    6.62e-05 &  2.00  &   4.50e-05  & \color{blue}{2.00}  &   2.64e-04 & \color{blue}{  2.00}  &   3.31e-05 &  2.00    \\
\hline
\multicolumn{9}{|>{\columncolor{mypink}}c|}{Second order CFO with $\beta =2$ }\\
\hline
h & $\|u_h-u \|_{0}$ & rate & $\|u_h-u\|_{1}$ & rate &$ \|q_h-q \|_{0}$& rate & $\|\lambda_h\|_{0}$ & rate   \\
\hline
1/8   &    6.13e-03  & -     &   8.10e-03 & \color{blue}{ -   }  &   5.50e-02  &\color{blue}{  -   }  &   2.80e-03 &  -    \\
1/16  &    8.71e-04  & 2.82  &   1.94e-03 & \color{blue}{ 2.06}  &   1.25e-02  &\color{blue}{  2.14}  &   4.20e-04 &  2.73 \\
1/32  &    1.17e-04  & 2.90  &   4.78e-04 & \color{blue}{ 2.02}  &   2.93e-03  &\color{blue}{  2.09}  &   5.73e-05 &  2.87 \\
1/64  &    1.51e-05  & 2.95  &   1.19e-04 & \color{blue}{ 2.01}  &   7.06e-04  &\color{blue}{  2.05}  &   7.49e-06 &  2.94 \\
1/128 &    1.93e-06  & 2.97  &   2.97e-05 & \color{blue}{ 2.00}  &   1.73e-04  &\color{blue}{  2.03}  &   9.57e-07 &  2.97 \\
\hline
\hline
\multicolumn{9}{|>{\columncolor{blue!20}}c|}{Third order CFO with $\beta =1$ }\\
\hline
h & $\|u_h-u \|_{0}$ & rate & $\|u_h-u\|_{1}$ & rate &$ \|q_h-q \|_{0}$& rate  & $\|\lambda_h\|_{0}$ & rate\\
\hline
1/8    &    2.31e-04 &   -    &   4.43e-04  &\color{blue}{  -  }  &   3.12e-03  &\color{blue}{  -  }  &   1.03e-04  &  -      \\
1/16   &    1.49e-05 &  3.96  &   5.45e-05  &\color{blue}{ 3.02}  &   3.43e-04  &\color{blue}{ 3.18}  &   6.78e-06  & 3.93    \\
1/32   &    9.40e-07 &  3.98  &   6.76e-06  &\color{blue}{ 3.01}  &   4.01e-05  &\color{blue}{ 3.10}  &   4.31e-07  & 3.98    \\
1/64   &    5.90e-08 &  3.99  &   8.42e-07  &\color{blue}{ 3.00}  &   4.86e-06  &\color{blue}{ 3.05}  &   2.71e-08  & 3.99    \\
1/128  &    3.70e-09 &  4.00  &   1.05e-07  &\color{blue}{ 3.00}  &   5.99e-07  &\color{blue}{ 3.02}  &   1.70e-09  & 4.00    \\
\hline
\multicolumn{9}{|>{\columncolor{blue!20}}c|}{Third order CFO with $\beta =2$ }\\
\hline
h & $\|u_h-u \|_{0}$ & rate & $\|u_h-u\|_{1}$ & rate &$ \|q_h-q \|_{0}$& rate  & $\|\lambda_h\|_{0}$ & rate  \\
\hline
1/8   &    9.64e-05 &   -    &   3.11e-04 & \color{blue}{   -  } &    3.15e-03  &\color{blue}{   -  }  &   3.04e-05  &  -    \\
1/16  &    4.67e-06 &  4.37  &   3.44e-05 & \color{blue}{  3.17} &    3.87e-04  &\color{blue}{  3.02}  &   1.14e-06  & 4.74  \\
1/32  &    2.43e-07 &  4.26  &   4.01e-06 & \color{blue}{  3.10} &    4.94e-05  &\color{blue}{  2.97}  &   3.91e-08  & 4.87  \\
1/64  &    1.37e-08 &  4.15  &   4.87e-07 & \color{blue}{  3.04} &    6.33e-06  &\color{blue}{  2.97}  &   1.28e-09  & 4.94  \\
1/128 &    8.17e-10 &  4.07  &   6.03e-08 & \color{blue}{  3.01} &    8.04e-07  &\color{blue}{  2.98}  &   4.06e-11  & 4.97  \\
\hline
\end{tabular}
\end{center}
\end{table}

\begin{table}[!h]
\scriptsize
\begin{center}
\caption{Errors between the CFO scheme and the Ritz Galerkin finite element method for Test Case 1.}\label{table11-FE-beta1}
\begin{tabular}{||c|cc|cc||}
\hline
\multicolumn{5}{|>{\columncolor{yellow!20}}c|}{ First order CFO with $\beta =1$ }\\
\hline
h &  $\|u_h-R_hu \|_{0}$ & rate & $\|u_h- R_hu\|_{1}$ & rate  \\
\hline
 1/8   &       3.45e-02  & \color{blue}{-   }  &   1.89e-02 &\color{blue}{  -   } \\
 1/16  &       9.24e-03  & \color{blue}{1.90}  &   4.91e-03 &\color{blue}{  1.94} \\
 1/32  &       2.35e-03  & \color{blue}{1.97}  &   1.24e-03 &\color{blue}{  1.98} \\
 1/64  &       5.91e-04  & \color{blue}{1.99}  &   3.11e-04 &\color{blue}{  2.00} \\
 1/128 &       1.48e-04  & \color{blue}{2.00}  &   7.79e-05 &\color{blue}{  2.00} \\
\hline
\multicolumn{5}{|>{\columncolor{yellow!20}}c|}{ First order CFO with $\beta =2$ }\\
\hline
h  & $\|u_h-R_hu \|_{0}$ & rate & $\|u_h- R_hu\|_{1}$ & rate  \\
\hline
 1/8   &    1.78e-02 &\color{blue}{  -   } &    9.61e-03 &\color{blue}{  -   }\\
 1/16  &    2.95e-03 &\color{blue}{  2.60} &    1.54e-03 &\color{blue}{  2.64}\\
 1/32  &    4.28e-04 &\color{blue}{  2.79} &    2.21e-04 &\color{blue}{  2.80}\\
 1/64  &    5.79e-05 &\color{blue}{  2.89} &    2.98e-05 &\color{blue}{  2.89}\\
 1/128 &    7.54e-06 &\color{blue}{  2.94} &    3.88e-06 &\color{blue}{  2.94}\\
\hline
\multicolumn{5}{|>{\columncolor{yellow!20}}c|}{ First order CFO with $\beta =3$ }\\
\hline
h  & $\|u_h-R_hu \|_{0}$ & rate & $\|u_h- R_hu\|_{1}$ & rate  \\
\hline
 1/8   &         4.79e-03 &\color{blue}{  -    } &   2.55e-03 &\color{blue}{  -   } \\
 1/16  &         3.41e-04 &\color{blue}{  3.81 } &   1.76e-04 &\color{blue}{  3.85} \\
 1/32  &         2.20e-05 &\color{blue}{  3.95 } &   1.13e-05 &\color{blue}{  3.96} \\
 1/64  &         1.39e-06 &\color{blue}{  3.99}  &   7.12e-07 &\color{blue}{  3.99} \\
 1/128 &         8.69e-08 &\color{blue}{  4.00}  &   4.46e-08 &\color{blue}{  4.00} \\
\hline
\multicolumn{5}{|>{\columncolor{mypink}}c|}{Second order CFO with $\beta =1$ }\\
\hline
h &   $\|u_h-R_hu \|_{0}$ & rate & $\|u_h- R_hu\|_{1}$ & rate  \\
\hline
1/8   &   1.67e-02 &\color{blue}{  -   } &     8.63e-03  & \color{blue}{-   }\\
1/16  &   4.22e-03 &\color{blue}{  1.98} &     2.16e-03  & \color{blue}{2.00}\\
1/32  &   1.06e-03 &\color{blue}{  2.00} &     5.40e-04  & \color{blue}{2.00}\\
1/64  &   2.65e-04 &\color{blue}{  2.00} &     1.35e-04  & \color{blue}{2.00}\\
1/128 &   6.62e-05 &\color{blue}{  2.00} &     3.38e-05  & \color{blue}{2.00}\\
\hline
\multicolumn{5}{|>{\columncolor{mypink}}c|}{Second order CFO with $\beta =2$ }\\
\hline
h &  $\|u_h-R_hu \|_{0}$ & rate & $\|u_h- R_hu\|_{1}$ & rate  \\
\hline
1/8   &    5.84e-03 &\color{blue}{  -   }  &    2.96e-03 &\color{blue}{  -   }\\
1/16  &    8.49e-04 &\color{blue}{  2.78}  &    4.28e-04 &\color{blue}{  2.79}\\
1/32  &    1.15e-04 &\color{blue}{  2.88}  &    5.79e-05 &\color{blue}{  2.89}\\
1/64  &    1.50e-05 &\color{blue}{  2.94}  &    7.54e-06 &\color{blue}{  2.94}\\
1/128 &    1.91e-06 &\color{blue}{  2.97}  &    9.62e-07 &\color{blue}{  2.97}\\
\hline
\multicolumn{5}{|>{\columncolor{mypink}}c|}{Second order CFO with $\beta =3$ }\\
\hline
h  & $\|u_h-R_hu \|_{0}$ & rate & $\|u_h- R_hu\|_{1}$ & rate  \\
\hline
1/8   &      9.51e-04  &\color{blue}{ -   }  &    4.81e-04 &\color{blue}{  -   }\\
1/16  &      6.18e-05  &\color{blue}{ 3.94}  &    3.11e-05 &\color{blue}{  3.95}\\
1/32  &      3.90e-06  &\color{blue}{ 3.99}  &    1.96e-06 &\color{blue}{  3.99}\\
1/64  &      2.44e-07  &\color{blue}{ 4.00}  &    1.23e-07 &\color{blue}{  4.00}\\
1/128 &      1.53e-08  &\color{blue}{ 4.00}  &    7.68e-09 &\color{blue}{  4.00}\\
\hline
\hline
\multicolumn{5}{|>{\columncolor{blue!20}}c|}{Third order CFO with $\beta =1$ }\\
\hline
h  & $\|u_h-R_hu \|_{0}$ & rate & $\|u_h- R_hu\|_{1}$ & rate  \\
\hline
1/8   &  2.15e-04  &\color{blue}{  -  }   &  4.57e-04 &\color{blue}{   -  }\\
1/16  &  1.39e-05  &\color{blue}{ 3.95}   &  5.52e-05 &\color{blue}{  3.00}\\
1/32  &  8.84e-07  &\color{blue}{ 3.98}   &  6.79e-06 &\color{blue}{  3.02}\\
1/64  &  5.55e-08  &\color{blue}{ 3.99}   &  8.45e-07 &\color{blue}{  3.01}\\
1/128 &  3.48e-09  &\color{blue}{ 4.00}   &  1.05e-07 &\color{blue}{  3.00}\\
\hline
\multicolumn{5}{|>{\columncolor{blue!20}}c|}{Third order CFO with $\beta =2$ }\\
\hline
h & $\|u_h-R_hu \|_{0}$ & rate & $\|u_h- R_hu\|_{1}$ & rate  \\
\hline
1/8   &       6.57e-05 &\color{blue}{   -  }  &   2.19e-04 &\color{blue}{   -  }\\
1/16  &       2.48e-06 &\color{blue}{  4.73}  &   1.75e-05 &\color{blue}{  3.65}\\
1/32  &       8.56e-08 &\color{blue}{  4.86}  &   1.28e-06 &\color{blue}{  3.78}\\
1/64  &       2.81e-09 &\color{blue}{  4.93}  &   8.71e-08 &\color{blue}{  3.87}\\
1/128 &       8.89e-11 &\color{blue}{  4.98}  &   5.71e-09 &\color{blue}{  3.93}\\
\hline
\multicolumn{5}{|>{\columncolor{blue!20}}c|}{Third order CFO with $\beta =3$ }\\
\hline
h & $\|u_h-R_hu \|_{0}$ & rate & $\|u_h- R_hu\|_{1}$ & rate  \\
\hline
1/8   &      1.14e-05 &\color{blue}{   -   } &   4.59e-05 &\color{blue}{   -  }\\
1/16  &      1.87e-07 &\color{blue}{  5.94 } &   1.52e-06 &\color{blue}{  4.91}\\
1/32  &      2.94e-09 &\color{blue}{ 5.99  } &   4.82e-08 &\color{blue}{  4.98}\\
1/64  &      4.61e-11 &\color{blue}{ 6.00  } &   1.51e-09 &\color{blue}{  5.00}\\
1/128 &      4.05e-12 &\color{blue}{  3.51 } &   4.72e-11 &\color{blue}{  5.00}\\
\hline
\end{tabular}
\end{center}
\end{table}

\subsection{Test Case 2: H\"older Continuous Coefficients} The coefficient matrix $\alpha$ in this test is given by
\begin{eqnarray}
\alpha=\begin{pmatrix}
        1+|x| & 0.5|x|^{\frac13}|y|^{\frac13} \\
        0.5|x|^{\frac13}|y|^{\frac13} & 1+|y|
        \end{pmatrix}
\end{eqnarray}
on the square domain $\Omega=(-1,1)^2$. The coefficient matrix is clearly non-smooth, but H\"older continuous. The right-hand side function and the Dirichlet boundary data are chosen to match the exact solution of $u(x,y)=\cos(\pi x)\cos(\pi y)$. Note that this example has been considered in \cite{Smears_SIAMJNA_2013,WangWang_2016}.

We use the CFO scheme \eqref{min-LagrangeForm2-1}-\eqref{min-LagrangeForm2-2} with $k=1,2,3 $ to approximate the above elliptic problem. The corresponding error and convergence information are reported in Tables \ref{table22-h12}, \ref{table22-FE-beta120-1} and \ref{table22-h-beta-12}. We observe that the convergence of the algorithm in $H^1$ norms has convergence order of $k <  1$. Moreover, for the flux on element edges, the convergence is also of optimal order of $k< 1$. In fact, the numerical results outperform the theory in $H^1$, as the coefficient $\alpha$ is non-smooth nor Lipschitz continuous on each element so that no convergence of order $k=1$ can be deduced from the theory.
Consider the domain $\Omega_2=(0.1,1)^2$, the coefficient matrix is smooth on this domain, the corresponding error and convergence results are reported in Tables \ref{table22-h12-part}, \ref{table22-FE-beta120-1-part} and \ref{table22-h-beta-12-part}.
We observe that the convergence of the algorithm in both the $L_2$ and the $H^1$ norms have optimal order of $k+1$ and $k$. Moreover, for the flux on element edges, the convergence is also of optimal order of $k$.
These numerical results support strongly the theoretical findings in the previous section.

\begin{table}[!h]
\scriptsize
\begin{center}
\caption{Error and convergence performance of the CFO scheme for Test Case 2 on domain $\Omega=(-1,1)^2$.}\label{table22-h12}
\begin{tabular}{|c|cc|cc|cc|cc|}
\hline
\multicolumn{9}{|>{\columncolor{yellow!20}}c|}{ First order CFO with $\beta =1$ }\\
%\multicolumn{3}{>{\columncolor{red}}l}{\color{white}\textsf{LONDON}}
\hline
h & $\|u_h-u \|_{0}$ & rate & $\|u_h-u\|_{1}$ & rate &$ \|q_h-q \|_{0}$& rate  & $\|\lambda_h\|_{0}$ & rate  \\
\hline
 1/8   &    6.03e-02  & -    &    9.90e-02  & \color{blue}{-   } &    9.39e-02 & \color{blue}{ -   } &    2.00e-01 &  -     \\
 1/16  &    1.50e-02  & 2.01 &    4.92e-02  & \color{blue}{1.01} &    4.69e-02 & \color{blue}{ 1.00} &    5.16e-02 &  1.95  \\
 1/32  &    3.63e-03  & 2.05 &    2.46e-02  & \color{blue}{1.00} &    2.35e-02 & \color{blue}{ 1.00} &    1.31e-02 &  1.97  \\
 1/64  &    8.62e-04  & 2.08 &    1.23e-02  & \color{blue}{1.00} &    1.18e-02 & \color{blue}{ 0.99} &    3.36e-03 &  1.97  \\
 1/128 &    2.01e-04  & 2.10 &    6.14e-03  & \color{blue}{1.00} &    5.95e-03 & \color{blue}{ 0.99} &    8.69e-04 &  1.95  \\
\hline
\multicolumn{9}{|>{\columncolor{yellow!20}}c|}{First order CFO with $\beta =2$ }\\
\hline
h & $\|u_h-u \|_{0}$ & rate & $\|u_h-u\|_{1}$ & rate &$ \|q_h-q \|_{0}$& rate  & $\|\lambda_h\|_{0}$ & rate   \\
\hline
 1/8   &    4.15e-02 &  -     &   9.82e-02 &  \color{blue}{-   }  &   9.33e-02 & \color{blue}{ -   } &    1.37e-02 &  -     \\
 1/16  &    9.82e-03 &  2.08  &   4.91e-02 &  \color{blue}{1.00}  &   4.68e-02 & \color{blue}{ 1.00} &    1.71e-03 &  3.00  \\
 1/32  &    2.37e-03 &  2.05  &   2.45e-02 &  \color{blue}{1.00}  &   2.35e-02 & \color{blue}{ 0.99} &    2.19e-04 &  2.96  \\
 1/64  &    5.83e-04 &  2.02  &   1.23e-02 &  \color{blue}{1.00}  &   1.18e-02 & \color{blue}{ 0.99} &    2.93e-05 &  2.90  \\
 1/128 &    1.45e-04 &  2.00  &   6.14e-03 &  \color{blue}{1.00}  &   5.95e-03 & \color{blue}{ 0.99} &    4.16e-06 &  2.82  \\
\hline
\hline
\multicolumn{9}{|>{\columncolor{mypink}}c|}{Second order CFO with $\beta =1$ }\\
\hline
h & $\|u_h-u \|_{0}$ & rate & $\|u_h-u\|_{1}$ & rate &$ \|q_h-q \|_{0}$& rate  & $\|\lambda_h\|_{0}$ & rate  \\
1/8   &    1.23e-02  & -      &  9.72e-03 & \color{blue}{ -   }  &   1.77e-01  & \color{blue}{ -   }  &   1.13e-02  & -      \\
1/16  &    2.85e-03  & 2.11   &  2.78e-03 & \color{blue}{ 1.81}  &   5.88e-02  & \color{blue}{ 1.59}  &   2.74e-03  & 2.04   \\
1/32  &    6.52e-04  & 2.13   &  1.08e-03 & \color{blue}{ 1.37}  &   2.71e-02  & \color{blue}{ 1.12}  &   6.33e-04  & 2.11   \\
1/64  &    1.62e-04  & 2.01   &  5.47e-04 & \color{blue}{ 0.98}  &   1.45e-02  & \color{blue}{ 0.89}  &   1.61e-04  & 1.97   \\
1/128 &    5.34e-05  & 1.60   &  3.03e-04 & \color{blue}{ 0.85}  &   8.12e-03  & \color{blue}{ 0.84}  &   5.54e-05  & 1.54   \\
\hline
\multicolumn{9}{|>{\columncolor{mypink}}c|}{Second order CFO with $\beta =2$ }\\
\hline
h & $\|u_h-u \|_{0}$ & rate & $\|u_h-u\|_{1}$ & rate &$ \|q_h-q \|_{0}$& rate & $\|\lambda_h\|_{0}$ & rate   \\
\hline
1/8   &    2.30e-03  & -      &  7.59e-03  &\color{blue}{ -   }  &   1.66e-01 & \color{blue}{  -   }  &   1.83e-03 &  -     \\
1/16  &    2.89e-04  & 2.99   &  1.90e-03  &\color{blue}{ 2.00}  &   5.61e-02 & \color{blue}{  1.56}  &   2.37e-04 &  2.95  \\
1/32  &    3.82e-05  & 2.92   &  4.75e-04  &\color{blue}{ 2.00}  &   2.61e-02 & \color{blue}{  1.11}  &   2.89e-05 &  3.03  \\
1/64  &    7.44e-06  & 2.36   &  1.21e-04  &\color{blue}{ 1.98}  &   1.40e-02 & \color{blue}{  0.89}  &   3.74e-06 &  2.95  \\
1/128 &    2.65e-06  & 1.49   &  3.36e-05  &\color{blue}{ 1.84}  &   7.82e-03 & \color{blue}{  0.84}  &   6.04e-07 &  2.63  \\
\hline
\hline
\multicolumn{9}{|>{\columncolor{blue!20}}c|}{Third order CFO with $\beta =1$ }\\
\hline
h & $\|u_h-u \|_{0}$ & rate & $\|u_h-u\|_{1}$ & rate &$ \|q_h-q \|_{0}$& rate  & $\|\lambda_h\|_{0}$ & rate\\
\hline
1/8   &    1.21e-03 &   -    &   2.58e-03 & \color{blue}{  -  }  &   8.06e-02 & \color{blue}{  -  }  &   1.44e-03  &  -   \\
1/16  &    4.68e-04 &  1.37  &   1.51e-03 & \color{blue}{ 0.77}  &   4.59e-02 & \color{blue}{ 0.81}  &   5.67e-04  & 1.35 \\
1/32  &    1.78e-04 &  1.40  &   8.63e-04 & \color{blue}{ 0.80}  &   2.59e-02 & \color{blue}{ 0.82}  &   2.18e-04  & 1.38 \\
1/64  &    6.85e-05 &  1.38  &   4.89e-04 & \color{blue}{ 0.82}  &   1.46e-02 & \color{blue}{ 0.83}  &   8.45e-05  & 1.37 \\
1/128 &    2.68e-05 &  1.36  &   2.76e-04 & \color{blue}{ 0.83}  &   8.21e-03 & \color{blue}{ 0.83}  &   3.31e-05  & 1.35 \\
\hline
\multicolumn{9}{|>{\columncolor{blue!20}}c|}{Third order CFO with $\beta =2$ }\\
\hline
h & $\|u_h-u \|_{0}$ & rate & $\|u_h-u\|_{1}$ & rate &$ \|q_h-q \|_{0}$& rate  & $\|\lambda_h\|_{0}$ & rate  \\
\hline
1/8   &    1.92e-04 &   -     &  4.26e-04  &\color{blue}{  -  }  &   7.61e-02 &\color{blue}{   -  }  &   4.29e-04 &   -   \\
1/16  &    2.58e-05 &  2.90   &  1.65e-04  &\color{blue}{ 1.37}  &   4.35e-02 &\color{blue}{  0.81}  &   1.03e-04 &  2.05 \\
1/32  &    1.93e-05 &  0.42   &  1.50e-04  &\color{blue}{ 0.13}  &   2.47e-02 &\color{blue}{  0.82}  &   2.28e-05 &  2.18 \\
1/64  &    1.14e-05 &  0.76   &  1.07e-04  &\color{blue}{ 0.49}  &   1.40e-02 &\color{blue}{  0.82}  &   4.77e-06 &  2.25 \\
1/128 &    5.35e-06 &  1.09   &  6.75e-05  &\color{blue}{ 0.67}  &   7.86e-03 &\color{blue}{  0.83}  &   9.75e-07 &  2.29 \\
\hline
\end{tabular}
\end{center}
\end{table}

\begin{table}[!h]
\scriptsize
\begin{center}
\caption{Error and convergence performance of the CFO scheme for Test Case 2 on domain $\Omega=(0.1,1)^2$.}\label{table22-h12-part}
\begin{tabular}{|c|cc|cc|cc|cc|}
\hline
\multicolumn{9}{|>{\columncolor{yellow!20}}c|}{ First order CFO with $\beta =1$ }\\
%\multicolumn{3}{>{\columncolor{red}}l}{\color{white}\textsf{LONDON}}
\hline
h & $\|u_h-u \|_{0}$ & rate & $\|u_h-u\|_{1}$ & rate &$ \|q_h-q \|_{0}$& rate  & $\|\lambda_h\|_{0}$ & rate  \\
\hline
 1/8   &    5.13e-02   &-        &1.04e-01  &\color{blue}{ -   }   &  1.05e-01  &\color{blue}{ -   }   &  1.60e-01 &  -     \\
 1/16  &    1.33e-02   &1.94     &5.22e-02  &\color{blue}{ 1.00}   &  5.22e-02  &\color{blue}{ 1.01}   &  4.15e-02 &  1.95  \\
 1/32  &    3.37e-03   &1.98     &2.61e-02  &\color{blue}{ 1.00}   &  2.60e-02  &\color{blue}{ 1.00}   &  1.05e-02 &  1.99  \\
 1/64  &    8.44e-04   &2.00     &1.29e-02  &\color{blue}{ 1.02}   &  1.28e-02  &\color{blue}{ 1.03}   &  2.65e-03 &  1.98  \\
 1/128 &    2.11e-04   &2.00     &6.42e-03  &\color{blue}{ 1.01}   &  6.33e-03  &\color{blue}{ 1.01}   &  6.65e-04 &  1.99  \\
\hline
\multicolumn{9}{|>{\columncolor{yellow!20}}c|}{First order CFO with $\beta =2$ }\\
\hline
h & $\|u_h-u \|_{0}$ & rate & $\|u_h-u\|_{1}$ & rate &$ \|q_h-q \|_{0}$& rate  & $\|\lambda_h\|_{0}$ & rate   \\
\hline
 1/8   &    4.60e-02  & -     &   1.04e-01 &\color{blue}{  -   }  &   1.05e-01 &\color{blue}{  -   }  &   2.94e-02 &  -    \\
 1/16  &    1.10e-02  & 2.06  &   5.21e-02 &\color{blue}{  1.00}  &   5.21e-02 &\color{blue}{  1.01}  &   3.57e-03 &  3.04 \\
 1/32  &    2.60e-03  & 2.09  &   2.61e-02 &\color{blue}{  1.00}  &   2.60e-02 &\color{blue}{  1.00}  &   4.24e-04 &  3.07 \\
 1/64  &    6.11e-04  & 2.09  &   1.29e-02 &\color{blue}{  1.02}  &   1.28e-02 &\color{blue}{  1.02}  &   4.98e-05 &  3.09 \\
 1/128 &    1.47e-04  & 2.05  &   6.42e-03 &\color{blue}{  1.01}  &   6.33e-03 &\color{blue}{  1.01}  &   6.01e-06 &  3.05 \\
\hline
\multicolumn{9}{|>{\columncolor{mypink}}c|}{Second order CFO with $\beta =1$ }\\
\hline
h & $\|u_h-u \|_{0}$ & rate & $\|u_h-u\|_{1}$ & rate &$ \|q_h-q \|_{0}$& rate  & $\|\lambda_h\|_{0}$ & rate  \\
\hline
1/8   &    1.34e-02 &  -    &    1.03e-02 &\color{blue}{  -   }  &   1.33e-01 & \color{blue}{  -   }  &   4.54e-03 &  -    \\
1/16  &    3.32e-03 &  2.01 &    2.57e-03 &\color{blue}{  2.00}  &   3.24e-02 & \color{blue}{  2.03}  &   1.22e-03 &  1.90 \\
1/32  &    8.28e-04 &  2.00 &    6.43e-04 &\color{blue}{  2.00}  &   8.02e-03 & \color{blue}{  2.02}  &   3.11e-04 &  1.97 \\
1/64  &    2.05e-04 &  2.01 &    1.61e-04 &\color{blue}{  2.00}  &   2.01e-03 & \color{blue}{  1.99}  &   8.06e-05 &  1.95 \\
1/128 &    5.10e-05 &  2.01 &    4.02e-05 &\color{blue}{  2.00}  &   5.04e-04 & \color{blue}{  2.00}  &   2.05e-05 &  1.98 \\
\hline
\multicolumn{9}{|>{\columncolor{mypink}}c|}{Second order CFO with $\beta =2$ }\\
\hline
h & $\|u_h-u \|_{0}$ & rate & $\|u_h-u\|_{1}$ & rate &$ \|q_h-q \|_{0}$& rate & $\|\lambda_h\|_{0}$ & rate   \\
\hline
1/8   &    6.63e-03 &  -      &  8.22e-03 &\color{blue}{  -   }  &   1.25e-01 & \color{blue}{  -   }   &  2.28e-03  & -     \\
1/16  &    1.03e-03 &  2.68   &  1.97e-03 &\color{blue}{  2.06}  &   3.00e-02 & \color{blue}{  2.05}   &  3.78e-04  & 2.59  \\
1/32  &    1.47e-04 &  2.81   &  4.81e-04 &\color{blue}{  2.03}  &   7.34e-03 & \color{blue}{  2.03}   &  5.49e-05  & 2.78  \\
1/64  &    1.90e-05 &  2.95   &  1.19e-04 &\color{blue}{  2.01}  &   1.84e-03 & \color{blue}{  2.00}   &  7.42e-06  & 2.89  \\
1/128 &    2.41e-06 &  2.98   &  2.97e-05 &\color{blue}{  2.00}  &   4.61e-04 & \color{blue}{  2.00}   &  9.63e-07  & 2.95  \\
\hline
\hline
\multicolumn{9}{|>{\columncolor{blue!20}}c|}{Third order CFO with $\beta =1$ }\\
\hline
h & $\|u_h-u \|_{0}$ & rate & $\|u_h-u\|_{1}$ & rate &$ \|q_h-q \|_{0}$& rate  & $\|\lambda_h\|_{0}$ & rate\\
\hline
1/8   &    2.92e-04 &   -    &   4.96e-04  &\color{blue}{  -  }  &   5.51e-03 & \color{blue}{  -  }  &   1.01e-04 &   -    \\
1/16  &    1.95e-05 &  3.91  &   6.14e-05  &\color{blue}{ 3.01}  &   6.22e-04 & \color{blue}{ 3.15}  &   6.79e-06 &  3.89  \\
1/32  &    1.25e-06 &  3.96  &   7.65e-06  &\color{blue}{ 3.01}  &   7.36e-05 & \color{blue}{ 3.08}  &   4.37e-07 &  3.96  \\
1/64  &    7.67e-08 &  4.02  &   9.47e-07  &\color{blue}{ 3.01}  &   9.18e-06 & \color{blue}{ 3.00}  &   2.80e-08 &  3.96  \\
1/128 &    4.75e-09 &  4.01  &   1.18e-07  &\color{blue}{ 3.01}  &   1.15e-06 & \color{blue}{ 3.00}  &   1.77e-09 &  3.98  \\
\hline
\multicolumn{9}{|>{\columncolor{blue!20}}c|}{Third order CFO with $\beta =2$ }\\
\hline
h & $\|u_h-u \|_{0}$ & rate & $\|u_h-u\|_{1}$ & rate &$ \|q_h-q \|_{0}$& rate  & $\|\lambda_h\|_{0}$ & rate  \\
\hline
1/8   &    1.32e-04 &   -    &   3.70e-04 &\color{blue}{   -  }  &   5.28e-03 & \color{blue}{  -   } &   3.34e-05 &   -   \\
1/16  &    6.46e-06 &  4.35  &   4.01e-05 &\color{blue}{  3.21}  &   6.52e-04 & \color{blue}{ 3.02 } &   1.36e-06 &  4.61 \\
1/32  &    3.24e-07 &  4.32  &   4.46e-06 &\color{blue}{  3.17}  &   8.38e-05 & \color{blue}{ 2.96 } &   4.99e-08 &  4.77 \\
1/64  &    1.69e-08 &  4.26  &   5.18e-07 &\color{blue}{  3.11}  &   1.10e-05 & \color{blue}{ 2.93 } &   1.67e-09 &  4.90 \\
1/128 &    9.60e-10 &  4.14  &   6.29e-08 &\color{blue}{  3.04}  &   1.42e-06 & \color{blue}{ 2.95 } &   5.38e-11 &  4.96 \\
\hline
\end{tabular}
\end{center}
\end{table}

\subsection{Test Case 3: Discontinuous Coefficients}
In this numerical test, the domain of the elliptic problem (\ref{EQ:EllipticTest}) is chosen as the unit square $\Omega=(0,1)^2$, and the coefficient $\alpha$ is given by
\begin{eqnarray}
\alpha=\left\{
       \begin{array}{ll}
       \begin{pmatrix}
        1 & 0 \\
        0 & 1
       \end{pmatrix}, \qquad \text{ if } x<0.5,\\[10pt]
       \begin{pmatrix}
        10 & 3 \\
        3  & 1
       \end{pmatrix}, \qquad \text{ if } x\geq 0.5,
       \end{array}\right.
\end{eqnarray}
which is clearly discontinuous along the vertical line of $x=\frac12$. With properly chosen data on the right-hand side function and the Dirichlet boundary value, the exact solution of (\ref{EQ:EllipticTest}) is given by
\begin{eqnarray}
u(x,y)=\left\{
       \begin{array}{ll}
          1-2y^2+4xy+6x+2y,\qquad  \text{ if } x<0.5,\\
          -2y^2+1.6xy-0.6x+3.2y+4.3, \qquad \text{ if } x\geq 0.5.
       \end{array}\right.
\end{eqnarray}

Tables \ref{table33-h12}, \ref{table33-h-beta-12} and \ref{table33-FE-beta120-1} illustrate the performance of the CFO scheme \eqref{min-LagrangeForm2-1}-\eqref{min-LagrangeForm2-2} when applied to the present test case. The results suggest an optimal order of convergence for the numerical approximation $u_h$ in the usual $H^1$ norm, which is in great consistency with the error estimate developed in the previous section. Likewise, the numerical approximation for the flux variable $q_h$ also has an optimal order of convergence, as predicted by the convergence theory. On the other hand, the convergence in $L^2$ for $u_h$ seems to be around $k=1.8$. Since the exact solution is polyoma of second order, for second order and third order CFO scheme, the errors are of machine precision.

\begin{table}[!h]
\scriptsize
\begin{center}
\caption{Error and convergence performance of the CFO scheme for Test Case 3.}\label{table33-h12}
\begin{tabular}{|c|cc|cc|cc|cc|}
\hline
\multicolumn{9}{|>{\columncolor{yellow!20}}c|}{ First order CFO with $\beta =1$ }\\
%\multicolumn{3}{>{\columncolor{red}}l}{\color{white}\textsf{LONDON}}
\hline
h & $\|u_h-u \|_{0}$ & rate & $\|u_h-u\|_{1}$ & rate &$ \|q_h-q \|_{0}$& rate  & $\|\lambda_h\|_{0}$ & rate  \\
\hline
1/8   &     6.24e-04  & -      &  1.64e-02  &\color{blue}{ -   }   &  2.92e-02  &\color{blue}{ -   }  &   1.04e-01  & -     \\
1/16  &     1.81e-04  & 1.78   &  8.17e-03  &\color{blue}{ 1.00}   &  1.29e-02  &\color{blue}{ 1.18}  &   2.70e-02  & 1.95  \\
1/32  &     5.13e-05  & 1.82   &  4.08e-03  &\color{blue}{ 1.00}   &  6.05e-03  &\color{blue}{ 1.09}  &   6.84e-03  & 1.98  \\
1/64  &     1.43e-05  & 1.85   &  2.04e-03  &\color{blue}{ 1.00}   &  2.95e-03  &\color{blue}{ 1.04}  &   1.72e-03  & 1.99  \\
1/128 &     3.91e-06  & 1.87   &  1.02e-03  &\color{blue}{ 1.00}   &  1.46e-03  &\color{blue}{ 1.01}  &   4.31e-04  & 2.00  \\
\hline
\multicolumn{9}{|>{\columncolor{yellow!20}}c|}{First order CFO with $\beta =2$ }\\
\hline
h & $\|u_h-u \|_{0}$ & rate & $\|u_h-u\|_{1}$ & rate &$ \|q_h-q \|_{0}$& rate  & $\|\lambda_h\|_{0}$ & rate   \\
\hline
1/8   &    5.62e-04  & -     &   1.63e-02  &\color{blue}{ -   }  &   2.84e-02 &\color{blue}{  -   }  &   1.74e-02 &  -      \\
1/16  &    1.37e-04  & 2.04  &   8.15e-03  &\color{blue}{ 1.00}  &   1.25e-02 &\color{blue}{  1.19}  &   2.12e-03 &  3.03   \\
1/32  &    3.40e-05  & 2.01  &   4.07e-03  &\color{blue}{ 1.00}  &   5.94e-03 &\color{blue}{  1.07}  &   2.52e-04 &  3.08   \\
1/64  &    8.86e-06  & 1.94  &   2.04e-03  &\color{blue}{ 1.00}  &   2.93e-03 &\color{blue}{  1.02}  &   2.98e-05 &  3.08   \\
1/128 &    2.33e-06  & 1.93  &   1.02e-03  &\color{blue}{ 1.00}  &   1.46e-03 &\color{blue}{  1.01}  &   3.56e-06 &  3.06   \\
\hline
\multicolumn{9}{|>{\columncolor{mypink}}c|}{Second order CFO with $\beta =1$ }\\
\hline
h & $\|u_h-u \|_{0}$ & rate & $\|u_h-u\|_{1}$ & rate &$ \|q_h-q \|_{0}$& rate  & $\|\lambda_h\|_{0}$ & rate  \\
\hline
1/8   &    2.07e-14 &  -     &   6.44e-14  & -     &   4.23e-12  &   -      &  8.26e-13 &  -       \\
1/16  &    6.43e-14 &  -     &   2.88e-13  & -     &   1.52e-11  &   -      &  2.26e-12 &  -       \\
1/32  &    1.99e-13 &  -     &   9.74e-13  & -     &   4.54e-11  &   -      &  7.33e-12 &  -       \\
1/64  &    9.64e-13 &  -     &   4.35e-12  & -     &   2.14e-10  &   -      &  3.62e-11 &  -       \\
1/128 &    3.80e-12 &  -     &   1.53e-11  & -     &   8.11e-10  &   -      &  1.47e-10 &  -       \\
\hline
\multicolumn{9}{|>{\columncolor{mypink}}c|}{Second order CFO with $\beta =2$ }\\
\hline
h & $\|u_h-u \|_{0}$ & rate & $\|u_h-u\|_{1}$ & rate &$ \|q_h-q \|_{0}$& rate & $\|\lambda_h\|_{0}$ & rate   \\
\hline
1/8   &    3.60e-15 &  -    &    2.51e-14 &  -     &   1.41e-12  &   -     &   2.01e-14 &  -       \\
1/16  &    2.84e-14 &  -    &    1.14e-13 &  -     &   6.51e-12  &   -     &   7.61e-14 &  -       \\
1/32  &    7.33e-14 &  -    &    2.01e-13 &  -     &   1.46e-11  &   -     &   1.04e-13 &  -       \\
1/64  &    9.85e-14 &  -    &    3.82e-13 &  -     &   1.98e-11  &   -     &   6.06e-14 &  -       \\
1/128 &    4.16e-13 &  -    &    1.04e-12 &  -     &   7.31e-11  &   -     &   1.35e-13 &  -       \\
\hline
\hline
\multicolumn{9}{|>{\columncolor{blue!20}}c|}{Third order CFO with $\beta =1$ }\\
\hline
h & $\|u_h-u \|_{0}$ & rate & $\|u_h-u\|_{1}$ & rate &$ \|q_h-q \|_{0}$& rate  & $\|\lambda_h\|_{0}$ & rate\\
\hline
1/8   &    1.04e-13 &   -   &    5.25e-13 &   -   &    3.48e-11 &   -   &    2.44e-13 &   -    \\
1/16  &    8.25e-13 &   -   &    3.20e-12 &   -   &    1.53e-10 &   -   &    1.38e-12 &   -    \\
1/32  &    3.93e-12 &   -   &    1.15e-11 &   -   &    6.77e-10 &   -   &    6.70e-12 &   -    \\
1/64  &    1.44e-11 &   -   &    3.93e-11 &   -   &    2.45e-09 &   -   &    2.47e-11 &   -    \\
1/128 &    5.21e-11 &   -   &    1.38e-10 &   -   &    8.89e-09 &   -   &    9.06e-11 &   -    \\
\hline
\multicolumn{9}{|>{\columncolor{blue!20}}c|}{Third order CFO with $\beta =2$ }\\
\hline
h & $\|u_h-u \|_{0}$ & rate & $\|u_h-u\|_{1}$ & rate &$ \|q_h-q \|_{0}$& rate  & $\|\lambda_h\|_{0}$ & rate  \\
\hline
1/8   &    1.27e-13 &   -    &   3.17e-13 &   -   &    1.88e-11 &   -   &    1.80e-13 &   -      \\
1/16  &    2.25e-13 &   -    &   5.46e-13 &   -   &    3.27e-11 &   -   &    2.34e-13 &   -      \\
1/32  &    6.29e-13 &   -    &   1.38e-12 &   -   &    9.26e-11 &   -   &    4.70e-13 &   -      \\
1/64  &    5.25e-13 &   -    &   1.44e-12 &   -   &    9.33e-11 &   -   &    2.73e-13 &   -      \\
1/128 &    1.35e-12 &   -    &   3.28e-12 &   -   &    2.46e-10 &   -   &    4.20e-13 &   -      \\
\hline
\end{tabular}
\end{center}
\end{table}

\subsection{Test Case 4: Discontinuous Coefficients}
In this test case, the domain $\Omega=(-1,1)^2$ is split into four subdomains $\Omega = \bigcup_{i=1}^{4} \Omega_i$ by the $x$ and $y$ axis. The diffusion coefficient $\alpha$ is given by
\begin{eqnarray*}
\alpha= \begin{pmatrix}
        \alpha_i^x & 0 \\
        0 & \alpha_i^y
       \end{pmatrix}, \text{ if } (x,y) \in \Omega_i,\\
\end{eqnarray*}
and the exact solution is given by $u(x,y)=\alpha_i \sin(2\pi x)\sin(2\pi y)$. Here the values of the coefficient $\alpha_i^x,\,\alpha_i^y$ and $\alpha_i$ are specified in Table \ref{table6}. It is clear that the diffusion coefficient $\alpha$ is discontinuous across the lines $x = 0$ and $y = 0$.

\begin{table}[h]
\caption{Test Case 4: Parameter values for the diffusion coefficients and the exact solution.}\label{table6}
\begin{center}
\begin{tabular}{|l|l|}
\hline
&\\[-7pt]
$\alpha^x_4 = 0.1$ & $\alpha^x_3 = 1000$   \\[1pt]
$\alpha^y_4 =0.01$ & $\alpha^y_3 = 100 $   \\[1pt]
$\alpha_4   = 100$ & $\alpha_3   = 0.01$   \\[3pt]
\hline
&\\ [-7pt]
$\alpha^x_1 = 100$   & $\alpha^x_2 = 1  $\\[1pt]
$\alpha^y_1 = 10 $   & $\alpha^y_2 = 0.1$\\[1pt]
$\alpha_1   = 0.1$   & $\alpha_2   = 10 $\\ [3pt]
\hline
\end{tabular}
\end{center}
\end{table}

Tables \ref{table44-h12}, \ref{table44-h-beta-12} and \ref{table44-FE-beta120-1} present the numerical performance of the CFO scheme \eqref{min-LagrangeForm2-1}-\eqref{min-LagrangeForm2-2} when applied to the present test case. The results suggest an optimal order of convergence for the numerical approximation $u_h$ in the usual $H^1$ norm and the flux variable $q_h$ in $L^2$. The numerical results are in great consistency with the error estimate developed in the previous section.

\begin{table}[!h]
\scriptsize
\begin{center}
\caption{Error and convergence performance of the CFO scheme for Test Case 4.}\label{table44-h12}
\begin{tabular}{|c|cc|cc|cc|cc|}
\hline
\multicolumn{9}{|>{\columncolor{yellow!20}}c|}{ First order CFO with $\beta =1$ }\\
%\multicolumn{3}{>{\columncolor{red}}l}{\color{white}\textsf{LONDON}}
\hline
h & $\|u_h-u \|_{0}$ & rate & $\|u_h-u\|_{1}$ & rate &$ \|q_h-q \|_{0}$& rate  & $\|\lambda_h\|_{0}$ & rate  \\
\hline
1/8   &    1.97e-01   &-     &   2.04e-01 &\color{blue}{  -   }  &   3.46e-01  &\color{blue}{ -   }  &   3.35e+00 &  -    \\
1/16  &    5.49e-02   &1.84  &   9.93e-02 &\color{blue}{  1.04}  &   1.65e-01  &\color{blue}{ 1.07}  &   1.19e+00 &  1.50 \\
1/32  &    1.40e-02   &1.97  &   4.92e-02 &\color{blue}{  1.01}  &   7.39e-02  &\color{blue}{ 1.16}  &   3.56e-01 &  1.74 \\
1/64  &    3.52e-03   &2.00  &   2.46e-02 &\color{blue}{  1.00}  &   3.42e-02  &\color{blue}{ 1.11}  &   9.72e-02 &  1.87 \\
1/128 &    8.78e-04   &2.00  &   1.23e-02 &\color{blue}{  1.00}  &   1.65e-02  &\color{blue}{ 1.05}  &   2.52e-02 &  1.95 \\
\hline
\multicolumn{9}{|>{\columncolor{yellow!20}}c|}{First order CFO with $\beta =2$ }\\
\hline
h & $\|u_h-u \|_{0}$ & rate & $\|u_h-u\|_{1}$ & rate &$ \|q_h-q \|_{0}$& rate  & $\|\lambda_h\|_{0}$ & rate   \\
\hline
1/8   &    1.68e-01 &  -    &    2.00e-01 &\color{blue}{  -    } &   3.30e-01 &\color{blue}{  -   }  &   2.73e-01  & -    \\
1/16  &    4.52e-02 &  1.90 &    9.87e-02 &\color{blue}{  1.02 } &   1.56e-01 &\color{blue}{  1.09}  &   4.64e-02  & 2.56 \\
1/32  &    1.15e-02 &  1.98 &    4.91e-02 &\color{blue}{  1.01 } &   7.00e-02 &\color{blue}{  1.15}  &   6.44e-03  & 2.85 \\
1/64  &    2.87e-03 &  2.00 &    2.46e-02 &\color{blue}{  1.00 } &   3.31e-02 &\color{blue}{  1.08}  &   7.89e-04  & 3.03 \\
1/128 &    7.18e-04 &  2.00 &    1.23e-02 &\color{blue}{  1.00 } &   1.62e-02 &\color{blue}{  1.03}  &   9.02e-05  & 3.13 \\
\hline
\multicolumn{9}{|>{\columncolor{mypink}}c|}{Second order CFO with $\beta =1$ }\\
\hline
h & $\|u_h-u \|_{0}$ & rate & $\|u_h-u\|_{1}$ & rate &$ \|q_h-q \|_{0}$& rate  & $\|\lambda_h\|_{0}$ & rate  \\
\hline
1/8   &    9.60e-02 &   -    &   7.38e-02 &\color{blue}{   -  }  &   5.46e+00  &\color{blue}{   -  }   &  1.55e-01  &  -   \\
1/16  &    2.61e-02 &  1.88  &   2.02e-02 &\color{blue}{  1.87}  &   1.41e+00  &\color{blue}{  1.95}   &  1.79e-02  & 3.11 \\
1/32  &    6.94e-03 &  1.91  &   4.71e-03 &\color{blue}{  2.10}  &   3.56e-01  &\color{blue}{  1.99}   &  2.03e-03  & 3.13 \\
1/64  &    1.78e-03 &  1.97  &   1.08e-03 &\color{blue}{  2.13}  &   8.87e-02  &\color{blue}{  2.00}   &  2.42e-04  & 3.07 \\
1/128 &    4.47e-04 &  1.99  &   2.59e-04 &\color{blue}{  2.06}  &   2.21e-02  &\color{blue}{  2.00}   &  2.99e-05  & 3.02 \\
\hline
\multicolumn{9}{|>{\columncolor{mypink}}c|}{Second order CFO with $\beta =2$ }\\
\hline
h & $\|u_h-u \|_{0}$ & rate & $\|u_h-u\|_{1}$ & rate &$ \|q_h-q \|_{0}$& rate & $\|\lambda_h\|_{0}$ & rate   \\
\hline
\hline
1/8   &    8.39e-03 &   -    &   3.17e-02 &\color{blue}{   -  } &    7.30e+00 &\color{blue}{    -  }   &  2.47e-02 &   -    \\
1/16  &    9.67e-04 &  3.12  &   7.79e-03 &\color{blue}{  2.02} &    1.85e+00 &\color{blue}{   1.98}   &  3.64e-03 &  2.76  \\
1/32  &    1.17e-04 &  3.05  &   1.92e-03 &\color{blue}{  2.02} &    4.57e-01 &\color{blue}{   2.02}   &  4.94e-04 &  2.88  \\
1/64  &    1.44e-05 &  3.02  &   4.76e-04 &\color{blue}{  2.01} &    1.13e-01 &\color{blue}{   2.02}   &  6.55e-05 &  2.92  \\
1/128 &    1.80e-06 &  3.01  &   1.19e-04 &\color{blue}{  2.00} &    2.78e-02 &\color{blue}{   2.02}   &  8.62e-06 &  2.93  \\
\hline
\multicolumn{9}{|>{\columncolor{blue!20}}c|}{Third order CFO with $\beta =1$ }\\
\hline
h & $\|u_h-u \|_{0}$ & rate & $\|u_h-u\|_{1}$ & rate &$ \|q_h-q \|_{0}$& rate  & $\|\lambda_h\|_{0}$ & rate\\
\hline
1/8   &    1.16e-03 &    -   &   2.56e-03  &\color{blue}{   - } &    1.65e+00 &\color{blue}{    - }   &  1.64e-02  &   -   \\
1/16  &    6.74e-05 &  4.10  &   2.84e-04  &\color{blue}{ 3.17} &    3.12e-01 &\color{blue}{  2.40}   &  1.17e-03  & 3.81  \\
1/32  &    4.05e-06 &  4.06  &   3.33e-05  &\color{blue}{ 3.10} &    5.75e-02 &\color{blue}{  2.44}   &  7.77e-05  & 3.91  \\
1/64  &    2.50e-07 &  4.01  &   4.06e-06  &\color{blue}{ 3.04} &    9.60e-03 &\color{blue}{  2.58}   &  5.13e-06  & 3.92  \\
1/128 &    1.57e-08 &  4.00  &   5.03e-07  &\color{blue}{ 3.01} &    1.30e-03 &\color{blue}{  2.89}   &  3.38e-07  & 3.92  \\
\hline
\multicolumn{9}{|>{\columncolor{blue!20}}c|}{Third order CFO with $\beta =2$ }\\
\hline
1/8   &    9.66e-04 &    -   &   2.26e-03 &\color{blue}{    - }  &   1.05e+00 &\color{blue}{    - } &    1.33e-03 &    -   \\
1/16  &    5.53e-05 &  4.13  &   2.60e-04 &\color{blue}{  3.12}  &   1.36e-01 &\color{blue}{  2.96} &    4.79e-05 &  4.79  \\
1/32  &    3.29e-06 &  4.07  &   3.14e-05 &\color{blue}{  3.05}  &   1.23e-02 &\color{blue}{  3.47} &    1.64e-06 &  4.87  \\
1/64  &    2.02e-07 &  4.03  &   3.87e-06 &\color{blue}{  3.02}  &   1.17e-03 &\color{blue}{  3.39} &    5.38e-08 &  4.93  \\
1/128 &    1.25e-08 &  4.01  &   4.81e-07 &\color{blue}{  3.01}  &   1.36e-04 &\color{blue}{  3.11} &    1.73e-09 &  4.96  \\
\hline
\end{tabular}
\end{center}
\end{table}

\subsection{The Lagrange multiplier $\lambda_h$}
The CFO algorithms involve two essential ideas in flux approximation: (1) the satisfaction of the mass conservation equation, and (2) the minimization of the object function $J_r(v,p)$ defined as in (\ref{EQ:functional}). As the value of the PDE coefficients may vary from element to element, one may modify the object function as follows by placing a weight $\tau_T$ on each element:
\begin{equation}\label{EQ:functional-new}
J_r^\beta(p,v):= \frac{1}{r}\sum_{T\in\T_h} \sum_{e\subset\pT} \tau_D h^\beta_D \int_e |p+\alpha\nabla v \cdot \bn_e |^r ds.
\end{equation}
For the CFO algorithm \eqref{EQ:argmin}, it can be seen from \eqref{min-LagrangeForm2-1}-\eqref{min-LagrangeForm2-2} that the weight parameter $\tau_D$ is automatically adjusted by the Lagrange multiplier $\lambda_h$ on each element, as the weight $\tau_D$ can be easily absorbed by $\lambda_h$ through the same scaling on each element. Therefore, the CFO algorithm is quite robust in the minimization part.

Figure \ref{fig:lambda} (or \ref{fig:lambda2order} and \ref{fig:lambda3order}) shows the surface plot of the Lagrange multiplier $\lambda_h$ obtained from first order (or second and third order) CFO scheme for each test case. It can be seen that $\lambda_h$ is quite sensitive to the continuity and smoothness of the true solution. Figure \ref{fig:lambda1order-error-l2}, \ref{fig:lambda1order-error-l2-test2}-\ref{fig:lambda1order-error-l2-test4} compare the square of $\lambda_h$ with $\|u_h-u \|^2_{0}$ norm for the first order scheme. We observe that when $\alpha$ is the identity matrix, $\lambda_h^2$ show exactly the same tendency with the $L_2$ error of between the exact solution and the numerical results from scheme \eqref{min-LagrangeForm2-1}-\eqref{min-LagrangeForm2-2}. The results for the second/third order scheme are included in the Appendix (Fig. \ref{fig:lambda2order-error-l2}-\ref{fig:lambda3order-error-l2-test4}), which show the same property as for the first order scheme. We conjecture that $\lambda_h$ would play the role of a posteriori error estimator in adaptive grid local refinements.

\begin{figure}[!h]
\centering
\subfigure[Smooth coefficient (left) and H\"older continuous coefficient (right)]{
\label{Fig.sub.1.lam}
\includegraphics [width=0.35\textwidth]{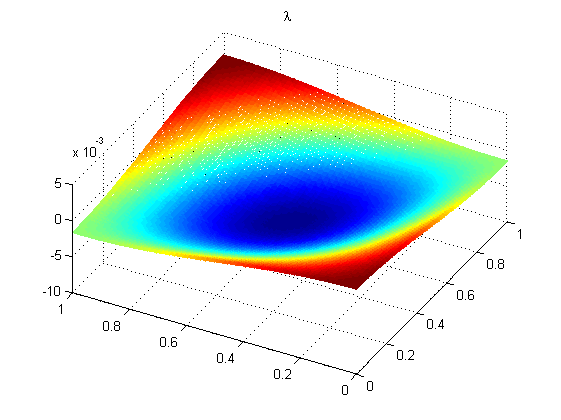}
\includegraphics [width=0.35\textwidth]{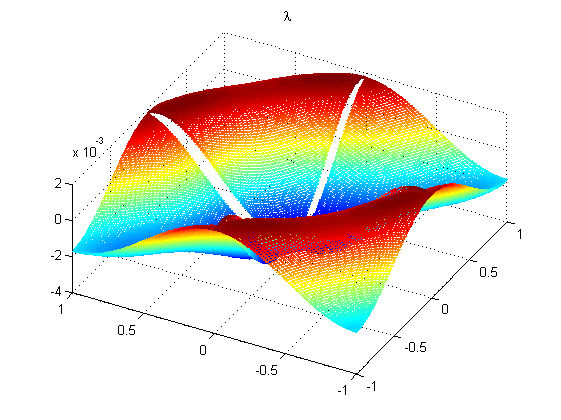}
}\\
\subfigure[Discontinuous coefficients: test case 3 (left) and test case 4 (right)]{
\label{Fig.sub.2.lam}
\includegraphics [width=0.35\textwidth]{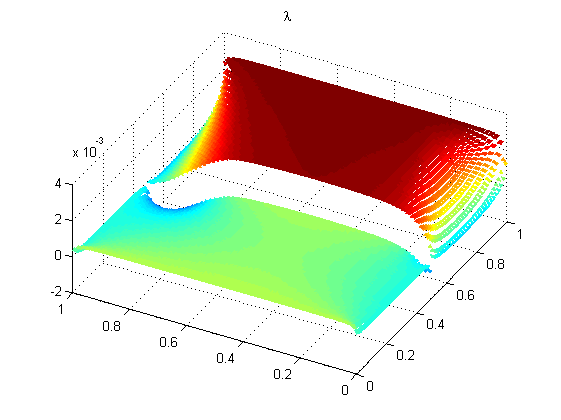}
\includegraphics [width=0.35\textwidth]{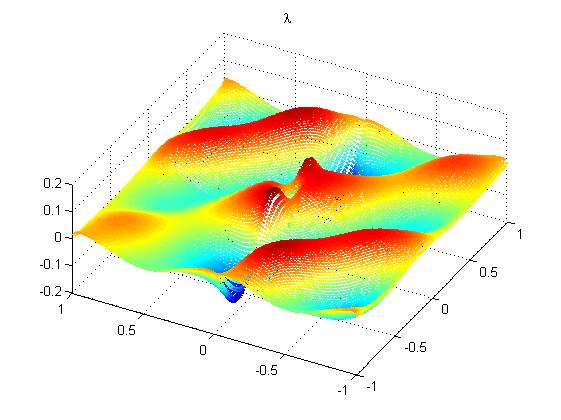}
}\\
\caption{The solution profile for the Lagrange multiplier $\lambda_h$ on a partition of size $64\times 64$
 arising from the first order CFO scheme (\ref{min-LagrangeForm2-1}-\ref{min-LagrangeForm2-2}).}
\label{fig:lambda}
\end{figure}
%%LMtestcase1%%%%%%%%%%%%%%%%%%%%%%%%%%%%%%%%%%%%%%%%%%%%%%%%%%%%%%%%%%%%%%%%%%%%%%%%%%%%%%%%%%%%%%%%%%%%%%%%%%%%%%%%%%%%%%%%%%%%%%%%%%%%%%%%%%%%%
\begin{figure}
\centering
\subfigure[$\lambda_h^2$ (left, stereogram above, plane diagram below) and $\|u_h-u \|^2_{0}$ (right, stereogram above, plane diagram below) ]{
\label{Fig.sub.1.lam1order-error}
\includegraphics [width=0.95\textwidth]{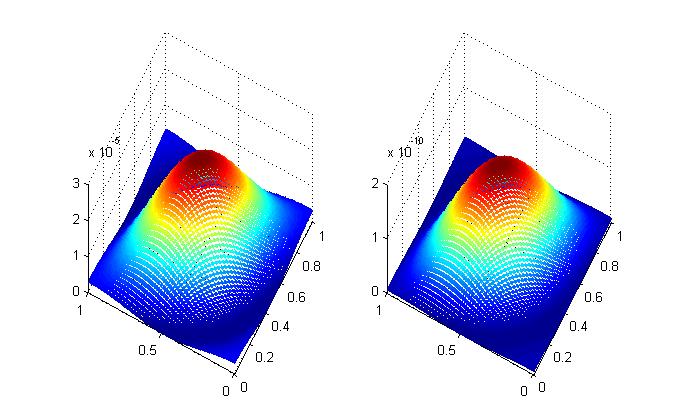}
}\\
\includegraphics [width=0.95\textwidth]{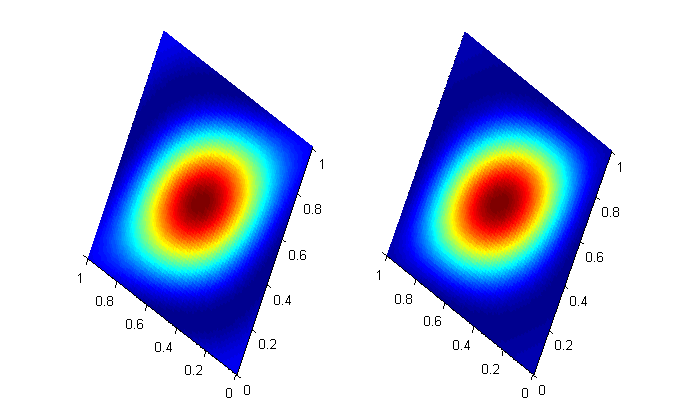}
\caption{The solution profile for the square of Lagrange multiplier $\lambda_h^2$ on a partition of size $64\times 64$
 arising from the first order CFO scheme (\ref{min-LagrangeForm2-1}-\ref{min-LagrangeForm2-2}) and the $L_2$ error of the primal variable $\|u_h-u \|^2_{0}$ for test case 1.}
\label{fig:lambda1order-error-l2}
\end{figure}

\subsection{A Two-Phase Flow in Porous Media}
We consider a simplified two-phase flow problem in porous media which has been studied in several existing literatures \cite{Bush_SIAMJSC_2013,Bush_JCAM_2014}.
It seeks a saturation function $S$ and fluid pressure $p$ satisfying
\begin{eqnarray}
-\nabla\cdot(\lambda(S)\kappa(x,y) \nabla p)&=&0, \qquad {\rm in}\  \Omega=(0,1)^2, \label{two-phase-elliptic}\\
\frac{\partial S}{\partial t} +\d(\bv f(S))&=&0,\qquad  t>0, \label{two-phase-transport}
\end{eqnarray}
with the boundary condition
\begin{eqnarray}
  p(0,y)=1, & & \; p(1,y)=0,\label{eq:001}\\
  \bv(x,0)\cdot \bn = 0, && \;  \bv(x,1)\cdot \bn=0,\label{eq:002}
\end{eqnarray}
for the fluid pressure $p$ and the following initial and boundary conditions for
the saturation:
\begin{eqnarray}
S(0,y,t)&=&1,\qquad t \geq 0, \ y\in (0,1),\\
S(x,y,0) &=& 0,\qquad (x,y)\in (0,1)^2.
\end{eqnarray}
Here in (\ref{two-phase-transport}) and the boundary condition \eqref{eq:002}, $\bv=-\lambda(S)\kappa(x,y) \nabla p$ is the Darcy's velocity of the fluid, $\kappa(x,y)$ is the permeability of the porous media, $f(S)$ is the fractional flow function, and $\lambda(S)$ is the total mobility. In our numerical study, we consider a permeability profile (Figure \ref{fig:permeability}) generated using the experimental data from the SPE comparative solution project \cite{Christie_SPE_2001}.
The permeability profiles in Figure \ref{fig:permeability} are plotted in a logarithmic scale, it shows that the permeability coefficients is highly heterogenous.

\begin{figure}
\centering
\includegraphics [width=0.45\textwidth]{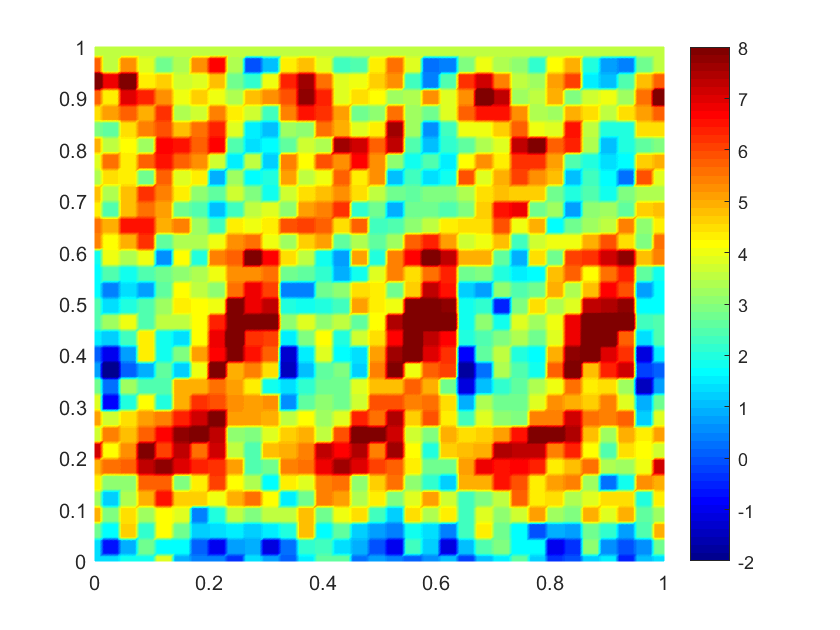}
\caption{Permeability Data obtained from the SPE comparative solution project \cite{Christie_SPE_2001} plotted in the logarithmic scale.}
\label{fig:permeability}
\end{figure}

We use a classical operator splitting technique \cite{Aziz_Chapman_1979} to solve the above system.
That is, we substitute the saturation at the previous time step into \eqref{two-phase-elliptic} to compute the pressure $p$, the Darcy's velocity $\bv$ and the locally conservative flux by using the CFO algorithms. Then we solve the transport equation \eqref{two-phase-transport} by an explicit time stepping scheme (see \cite{LiuWang_SINUM_2017} for details).

The saturation profiles of the two-phase flow for the permeability profile Fig. \ref{fig:permeability} at time $T= 0.02$  are shown in Figure \ref{fig:saturation}.
These profiles are obtained by using the 1st, 2nd and 3rd order CFO schemes respectively.
For each scheme, the computation is performed on meshes of $32\times32$, $64\times64$ and $128\times128 $ partitions.
These results match the physical tendency as described in \cite{Bush_SIAMJSC_2013, EfendievGinting_JCP_2006,LiuWang_SINUM_2017}.
Since the permeability coefficients is highly heterogenous, which varies from $e^{-2}$ to $e^{9}$,
it is very difficult to obtain convergence with our limiting mesh refinement.
The results show obviously that the higher order schemes allow to better capture the physical tendency, especially on coarse meshes.

\begin{figure}
\centering
~~~~~~~~~~~~~~$32\times32$~~~~~~~~~~~~~~~~~~~~~~~~~~~~$64\times64$~~~~~~~~~~~~~~~~~~~~~~~~~~~$128\times128$~~~~~~~~~~~~~~~\\
\vskip-1pt
\subfigure[Saturation obtained with the first order CFO scheme]{
\label{Fig.sat.1}
\includegraphics [width=0.315\textwidth]{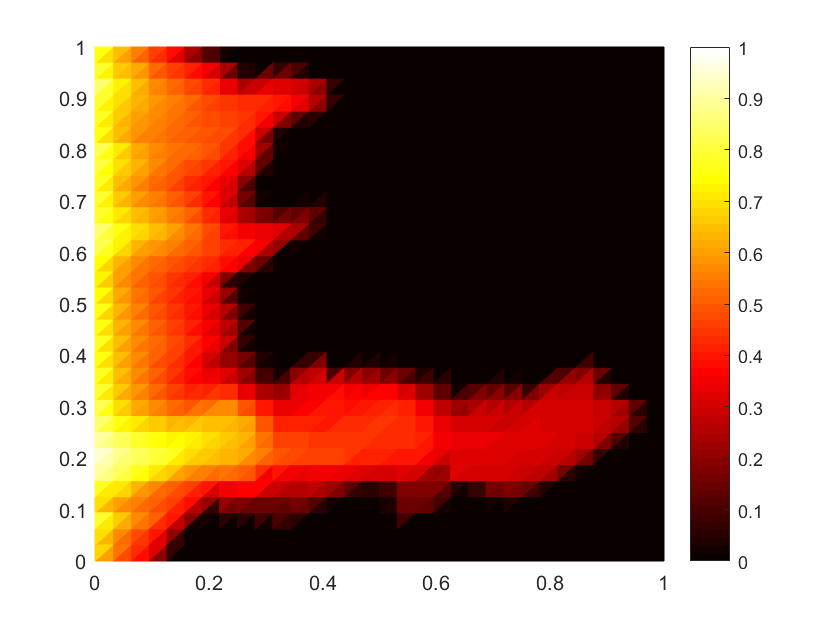}
\includegraphics [width=0.315\textwidth]{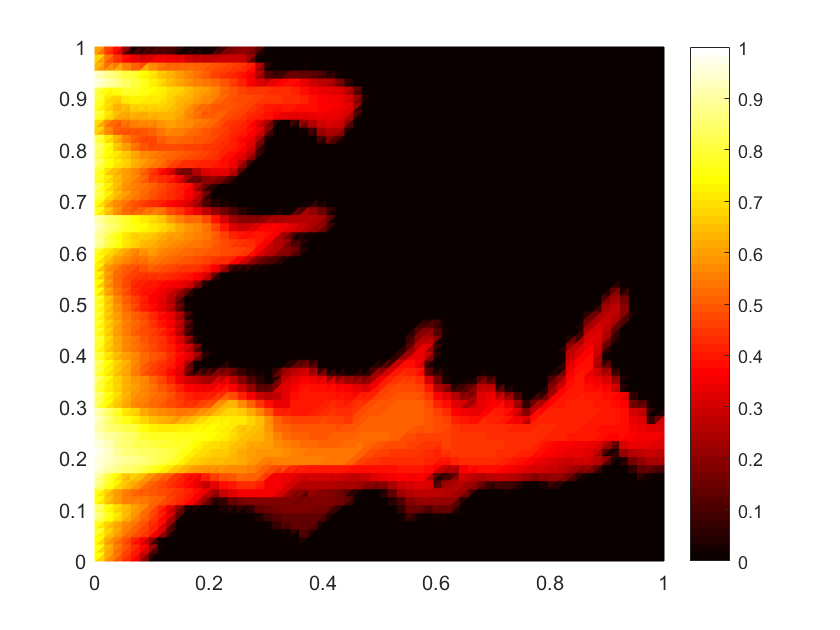}
\includegraphics [width=0.315\textwidth]{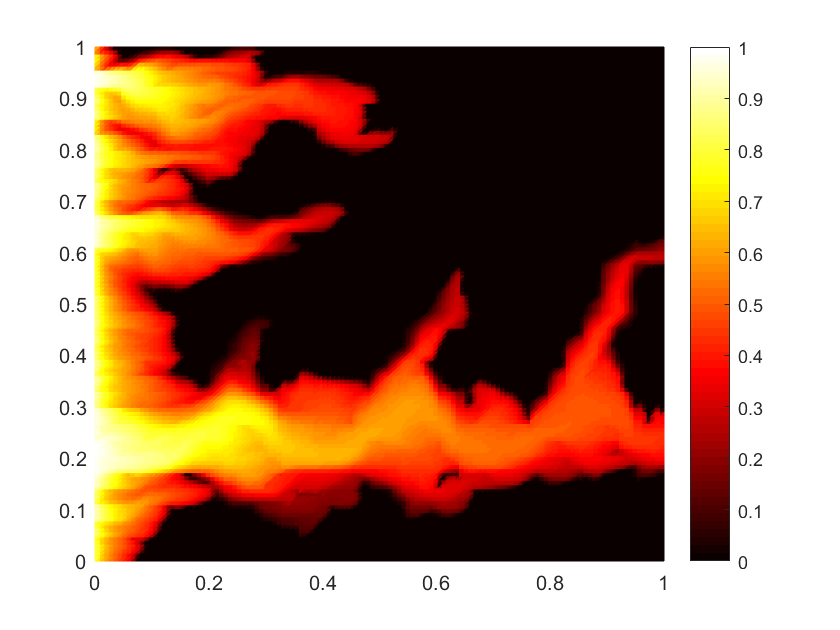}
}\\
~~~~~~~~~~~~~~$32\times32$~~~~~~~~~~~~~~~~~~~~~~~~~~~~$64\times64$~~~~~~~~~~~~~~~~~~~~~~~~~~~$128\times128$~~~~~~~~~~~~~~~\\
\vskip-1pt
\subfigure[Saturation obtained with the second order CFO scheme]{
\label{Fig.sat.2}
\includegraphics [width=0.315\textwidth]{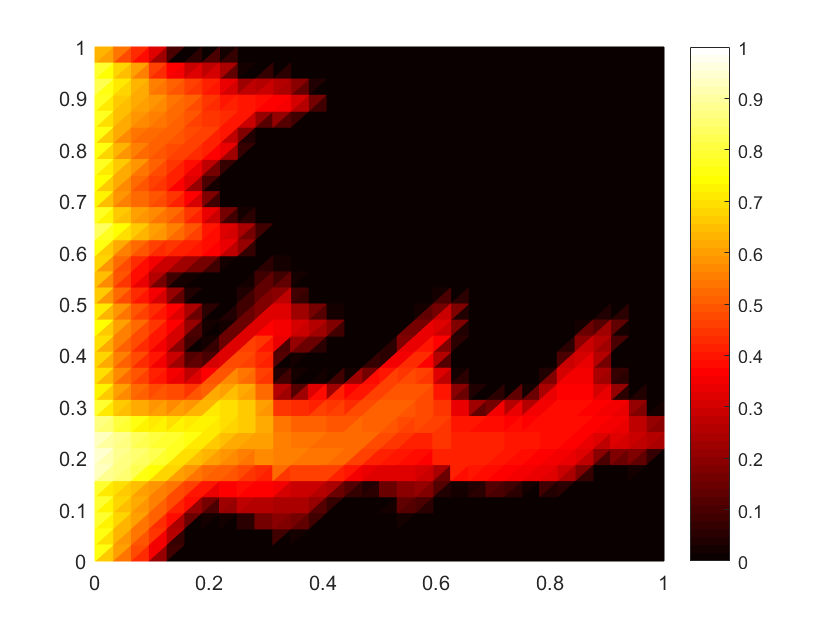}
\includegraphics [width=0.315\textwidth]{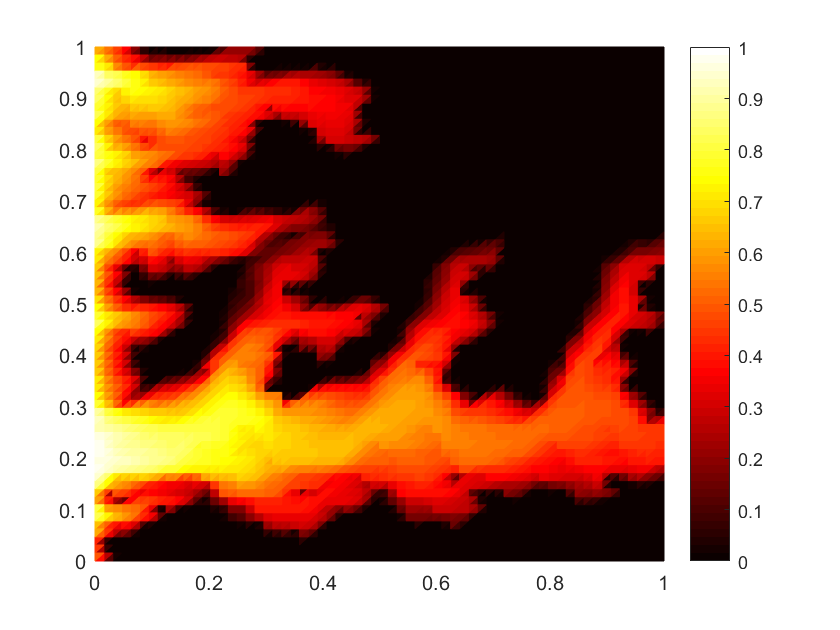}
\includegraphics [width=0.315\textwidth]{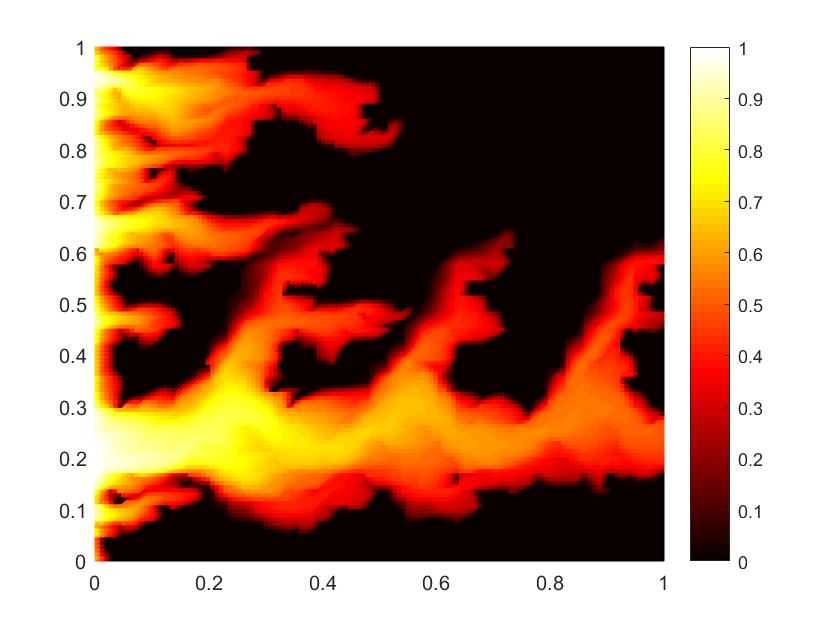}
}
\\
~~~~~~~~~~~~~~$32\times32$~~~~~~~~~~~~~~~~~~~~~~~~~~~~$64\times64$~~~~~~~~~~~~~~~~~~~~~~~~~~~$128\times128$~~~~~~~~~~~~~~~\\
\vskip-1pt
\subfigure[Saturation obtained with the third order CFO scheme]{
\label{Fig.sat.3}
\includegraphics [width=0.315\textwidth]{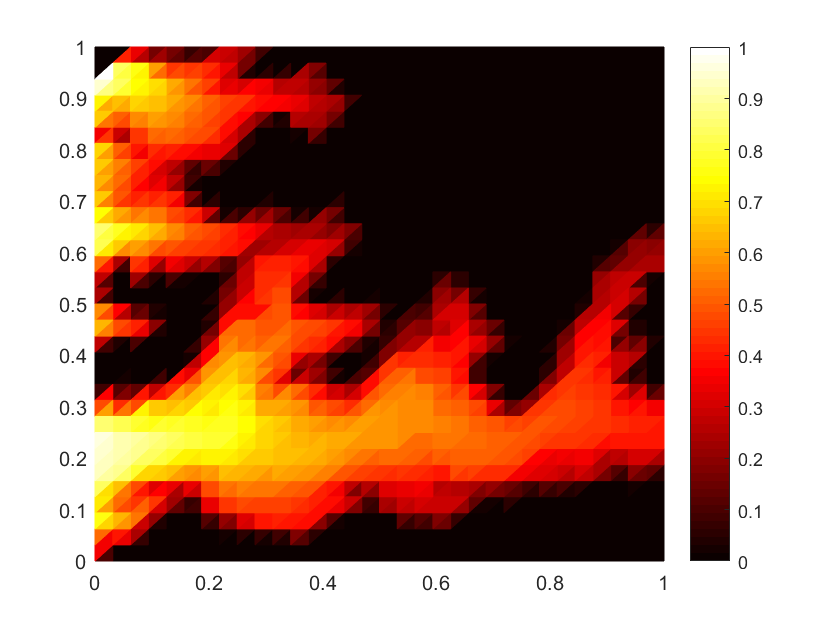}
\includegraphics [width=0.315\textwidth]{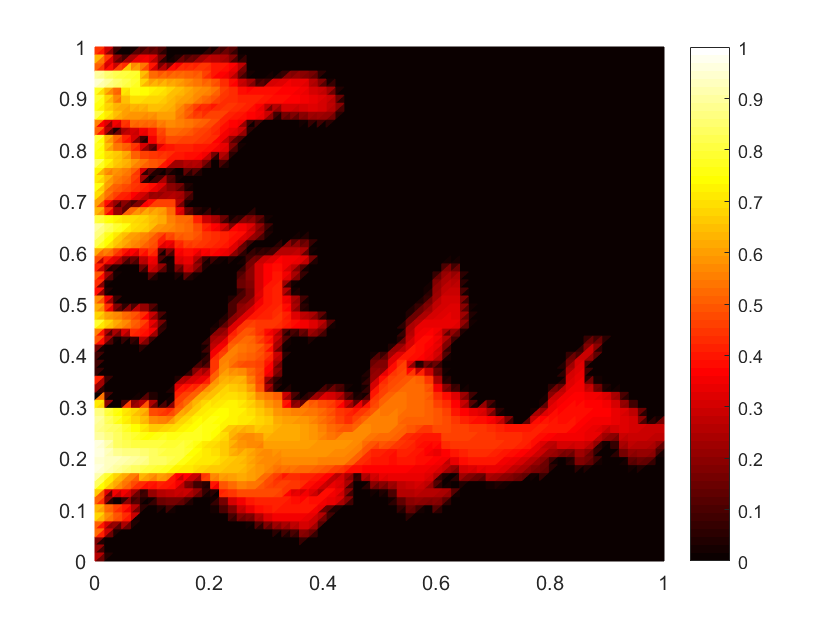}
\includegraphics [width=0.315\textwidth]{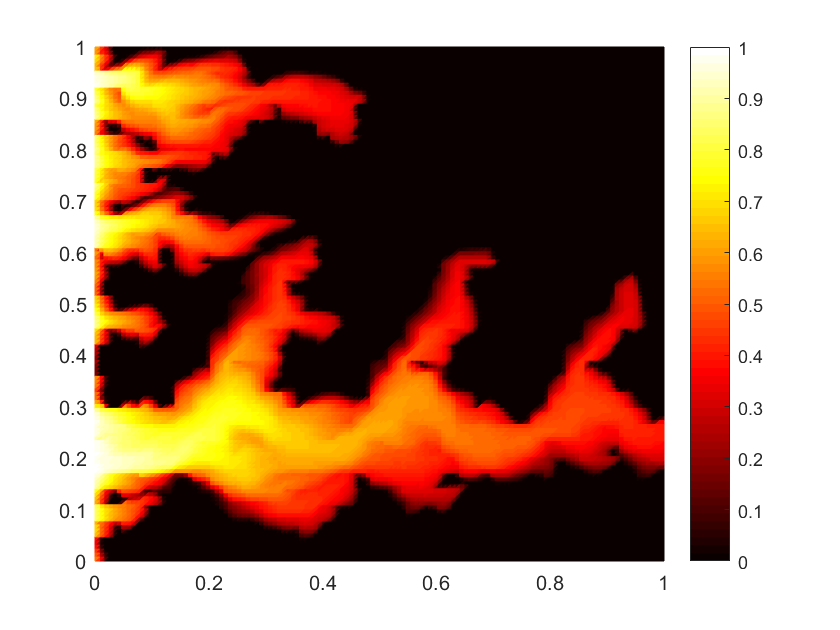}
}
\caption{Saturation profiles at time levels $T=0.02$ for the two-phase flow test case with the permeability profile Fig.\ref{fig:permeability}.}
\label{fig:saturation}
\end{figure}

\section*{Acknowledgements}
The authors are grateful to Todd Arbogast, Malgorzata Peszynska, Ralph Showalter, Junping Wang and Son-Young Yi for a helpful discussion of this work during the SIAM 2017 Mathematical and Computational Geoscience conference in Erlangen Germany. This discussion has led to the name of ``flux optimization" for the numerical scheme developed in this paper.

%\newpage

\clearpage
\appendix
\section{Numerical results}

\subsection{Numerical errors \& Convergence rate}\label{ErrorTables}
~
%case1%%%%%%%%%%%%%%%%%%%%%%%%%%%%%%%%%%%%%%%%%%%%%%%%%%%%%%%%%%%%%%%%%%%%%%%%%%%%%%%%%%%%%%%%%%%%%%%%%%%%%%%%%%
\begin{table}[!ht]
\scriptsize
\begin{center}
\caption{Error and convergence performance of the CFO scheme for Test Case 1 with $\beta=0,-1$.}\label{table11-h-beta-12}
\begin{tabular}{|c|cc|cc|cc|cc|}
\hline
\multicolumn{9}{|>{\columncolor{yellow!20}}c|}{ First order CFO with $\beta =0$ }\\
%\multicolumn{3}{>{\columncolor{red}}l}{\color{white}\textsf{LONDON}}
\hline
h & $\|u_h-u \|_{0}$ & rate & $\|u_h-u\|_{1}$ & rate &$ \|q_h-q \|_{0}$& rate  & $\|\lambda_h\|_{0}$ & rate  \\
\hline
1/8   &    8.12e-02  & -     &   1.03e-01  & \color{blue}{-   }   &  8.54e-02  &\color{blue}{ -   }  &   1.03e+00 &  -       \\
1/16  &    2.16e-02  & 1.91  &   4.98e-02  & \color{blue}{1.05}   &  4.08e-02  &\color{blue}{ 1.07}  &   5.41e-01 &  0.92    \\
1/32  &    5.50e-03  & 1.97  &   2.46e-02  & \color{blue}{1.02}   &  2.01e-02  &\color{blue}{ 1.02}  &   2.75e-01 &  0.98    \\
1/64  &    1.38e-03  & 1.99  &   1.23e-02  & \color{blue}{1.00}   &  1.00e-02  &\color{blue}{ 1.01}  &   1.38e-01 &  0.99    \\
1/128 &    3.47e-04  & 2.00  &   6.14e-03  & \color{blue}{1.00}   &  5.01e-03  &\color{blue}{ 1.00}  &   6.94e-02 &  1.00    \\
\hline
\multicolumn{9}{|>{\columncolor{yellow!20}}c|}{First order CFO with $\beta =-1$ }\\
\hline
\end{tabular}
\end{center}
\end{table}

\begin{table}[!ht]
\scriptsize
\begin{center}
\caption{Errors between the CFO scheme and the Ritz Galerkin finite element method for Test Case 1 with $\beta =0,-1$.}\label{table11-FE-beta0-1}
\begin{tabular}{|c|cc|cc|}
\hline
\multicolumn{5}{|>{\columncolor{yellow!20}}c|}{ First order CFO with $\beta =0$ }\\
\hline
h &  $\|u_h-R_hu \|_{0}$ & rate & $\|u_h- R_hu\|_{1}$ & rate  \\
\hline

1/8   &         4.14e-02  & \color{blue}{-   }  &   2.28e-02 &  \color{blue}{-   }\\
1/16  &         1.14e-02  & \color{blue}{1.86}  &   6.11e-03 &  \color{blue}{1.90}\\
1/32  &         2.95e-03  & \color{blue}{1.95}  &   1.57e-03 &  \color{blue}{1.96}\\
1/64  &         7.45e-04  & \color{blue}{1.98}  &   3.96e-04 &  \color{blue}{1.99}\\
1/128 &         1.87e-04  & \color{blue}{1.99}  &   9.93e-05 &  \color{blue}{1.99}\\
\hline
\multicolumn{5}{|>{\columncolor{yellow!20}}c|}{ First order CFO with $\beta =-1$ }\\
\hline
\end{tabular}
\end{center}
\end{table}
%case2%%%%%%%%%%%%%%%%%%%%%%%%%%%%%%%%%%%%%%%%%%%%%%%%%%%%%%%%%%%%%%%%%%%%%%%%%%%%%%%%%%%%%%%%%%%%%%%%%%%%%%%%%%
\begin{table}[!ht]
\scriptsize
\begin{center}
\caption{Error and convergence performance of the CFO scheme for Test Case 2 with $\beta=0,-1$.}\label{table22-h-beta-12}
\begin{tabular}{|c|cc|cc|cc|cc|}
\hline
\multicolumn{9}{|>{\columncolor{yellow!20}}c|}{ First order CFO with $\beta =0$ }\\
%\multicolumn{3}{>{\columncolor{red}}l}{\color{white}\textsf{LONDON}}
\hline
h & $\|u_h-u \|_{0}$ & rate & $\|u_h-u\|_{1}$ & rate &$ \|q_h-q \|_{0}$& rate  & $\|\lambda_h\|_{0}$ & rate  \\
\hline
1/8   &    6.95e-02 &  -      &  9.96e-02 &\color{blue}{  -   }   &  9.43e-02  &\color{blue}{ -   }  &   2.51e+00 &  -       \\
1/16  &    1.73e-02 &  2.00   &  4.93e-02 &\color{blue}{  1.01}   &  4.69e-02  &\color{blue}{ 1.01}  &   1.30e+00 &  0.95    \\
1/32  &    4.17e-03 &  2.06   &  2.46e-02 &\color{blue}{  1.00}   &  2.35e-02  &\color{blue}{ 1.00}  &   6.57e-01 &  0.98    \\
1/64  &    9.74e-04 &  2.10   &  1.23e-02 &\color{blue}{  1.00}   &  1.18e-02  &\color{blue}{ 0.99}  &   3.30e-01 &  0.99    \\
1/128 &    2.26e-04 &  2.11   &  6.15e-03 &\color{blue}{  1.00}   &  5.95e-03  &\color{blue}{ 0.99}  &   1.65e-01 &  1.00    \\
\hline
\multicolumn{9}{|>{\columncolor{yellow!20}}c|}{First order CFO with $\beta =-1$ }\\
\hline
\end{tabular}
\end{center}
\end{table}

\begin{table}[!ht]
\scriptsize
\begin{center}
\caption{Errors between the CFO scheme and the Ritz Galerkin finite element method for Test Case 2 with $\beta=1,2,0,-1$.}\label{table22-FE-beta120-1}
\begin{tabular}{|c|cc|cc|cc|cc|}
\hline
\multicolumn{5}{|>{\columncolor{yellow!20}}c|}{ First order CFO with $\beta =1$ }& \multicolumn{4}{|>{\columncolor{yellow!20}}c|}{ First order CFO with $\beta =0$ }\\
\hline
h &  $\|u_h-R_hu \|_{0}$ & rate & $\|u_h- R_hu\|_{1}$ & rate &  $\|u_h-R_hu \|_{0}$ & rate & $\|u_h- R_hu\|_{1}$ & rate \\
\hline
1/8   &    3.02e-02  & \color{blue}{-   } &    1.16e-02  &\color{blue}{ -   }  &      4.01e-02 &\color{blue}{  -   }  &   1.54e-02  &\color{blue}{ -   }\\
1/16  &    7.33e-03  & \color{blue}{2.04} &    3.04e-03  &\color{blue}{ 1.93}  &      9.89e-03 &\color{blue}{  2.02}  &   4.13e-03  &\color{blue}{ 1.90}\\
1/32  &    1.72e-03  & \color{blue}{2.09} &    1.09e-03  &\color{blue}{ 1.48}  &      2.32e-03 &\color{blue}{  2.09}  &   1.55e-03  &\color{blue}{ 1.42}\\
1/64  &    4.02e-04  & \color{blue}{2.10} &    5.55e-04  &\color{blue}{ 0.97}  &      5.39e-04 &\color{blue}{  2.11}  &   8.17e-04  &\color{blue}{ 0.92}\\
1/128 &    1.02e-04  & \color{blue}{1.97} &    3.14e-04  &\color{blue}{ 0.82}  &      1.40e-04 &\color{blue}{  1.95}  &   4.67e-04  &\color{blue}{ 0.81}\\
\hline
\multicolumn{5}{|>{\columncolor{yellow!20}}c|}{ First order CFO with $\beta =2$ }& \multicolumn{4}{|>{\columncolor{yellow!20}}c|}{ First order CFO with $\beta =-1$ }\\
\hline
\end{tabular}
\end{center}
\end{table}
%case2 domain (0.1,1)%%%%%%%%%%%%%%%%%%%%%%%%%%%%%%%%%%%%%%%%%%%%%%%%%%%%%%%%%%%%%%%%%%%%%%%%%%%%%%%%%%%%%%%%%%%%%%%%%%%%%%%%%%
\begin{table}[!ht]
\scriptsize
\begin{center}
\caption{Error and convergence performance of the CFO scheme for Test Case 2 with $\beta=0,-1$ on domain $\Omega=(0.1,1)^2$.}\label{table22-h-beta-12-part}
\begin{tabular}{|c|cc|cc|cc|cc|}
\hline
\multicolumn{9}{|>{\columncolor{yellow!20}}c|}{ First order CFO with $\beta =0$ }\\
%\multicolumn{3}{>{\columncolor{red}}l}{\color{white}\textsf{LONDON}}
\hline
h & $\|u_h-u \|_{0}$ & rate & $\|u_h-u\|_{1}$ & rate &$ \|q_h-q \|_{0}$& rate  & $\|\lambda_h\|_{0}$ & rate  \\
\hline
1/8   &    5.31e-02 &  -     &   1.05e-01 & \color{blue}{ -   }  &   1.05e-01 &\color{blue}{  -   }  &   8.14e-01  & -     \\
1/16  &    1.39e-02 &  1.93  &   5.22e-02 & \color{blue}{ 1.00}  &   5.22e-02 &\color{blue}{  1.01}  &   4.26e-01  & 0.94  \\
1/32  &    3.53e-03 &  1.98  &   2.61e-02 & \color{blue}{ 1.00}  &   2.60e-02 &\color{blue}{  1.01}  &   2.16e-01  & 0.98  \\
1/64  &    8.87e-04 &  1.99  &   1.29e-02 & \color{blue}{ 1.02}  &   1.28e-02 &\color{blue}{  1.03}  &   1.11e-01  & 0.95  \\
1/128 &    2.22e-04 &  2.00  &   6.42e-03 & \color{blue}{ 1.01}  &   6.33e-03 &\color{blue}{  1.01}  &   5.65e-02  & 0.98  \\
\hline
\multicolumn{9}{|>{\columncolor{yellow!20}}c|}{First order CFO with $\beta =-1$ }\\
\hline
\end{tabular}
\end{center}
\end{table}

\begin{table}[!ht]
\scriptsize
\begin{center}
\caption{Errors between the CFO scheme and the Ritz Galerkin finite element method for Test Case 2 with $\beta=1,2,0,-1$ on domain $\Omega=(0.1,1)^2$.}\label{table22-FE-beta120-1-part}
\begin{tabular}{|c|cc|cc|cc|cc|}
\hline
\multicolumn{5}{|>{\columncolor{yellow!20}}c|}{ First order CFO with $\beta =1$ }& \multicolumn{4}{|>{\columncolor{yellow!20}}c|}{ First order CFO with $\beta =0$ }\\
\hline
h &  $\|u_h-R_hu \|_{0}$ & rate & $\|u_h- R_hu\|_{1}$ & rate &  $\|u_h-R_hu \|_{0}$ & rate & $\|u_h- R_hu\|_{1}$ & rate \\
\hline
1/8   &   1.91e-02 & \color{blue}{ -   }  &   1.09e-02  & \color{blue}{-   }&         2.12e-02  &\color{blue}{ -   }  &   1.21e-02 &\color{blue}{  -   }\\
1/16  &   5.10e-03 & \color{blue}{ 1.91}  &   2.89e-03  & \color{blue}{1.91}&         5.75e-03  &\color{blue}{ 1.88}  &   3.27e-03 &\color{blue}{  1.89}\\
1/32  &   1.30e-03 & \color{blue}{ 1.97}  &   7.39e-04  & \color{blue}{1.97}&         1.48e-03  &\color{blue}{ 1.96}  &   8.42e-04 &\color{blue}{  1.96}\\
1/64  &   3.34e-04 & \color{blue}{ 1.96}  &   1.90e-04  & \color{blue}{1.96}&         3.82e-04  &\color{blue}{ 1.95}  &   2.18e-04 &\color{blue}{  1.95}\\
1/128 &   8.45e-05 & \color{blue}{ 1.98}  &   4.82e-05  & \color{blue}{1.98}&         9.69e-05  &\color{blue}{ 1.98}  &   5.55e-05 &\color{blue}{  1.98}\\
\hline
\multicolumn{5}{|>{\columncolor{yellow!20}}c|}{ First order CFO with $\beta =2$ }& \multicolumn{4}{|>{\columncolor{yellow!20}}c|}{ First order CFO with $\beta =-1$ }\\
\hline
\end{tabular}
\end{center}
\end{table}
%case3%%%%%%%%%%%%%%%%%%%%%%%%%%%%%%%%%%%%%%%%%%%%%%%%%%%%%%%%%%%%%%%%%%%%%%%%%%%%%%%%%%%%%%%%%%%%%%%%%%%%%%%%%%
\begin{table}[!ht]
\scriptsize
\begin{center}
\caption{Error and convergence performance of the CFO scheme for Test Case 3 with $\beta=0,-1$.}\label{table33-h-beta-12}
\begin{tabular}{|c|cc|cc|cc|cc|}
\hline
\multicolumn{9}{|>{\columncolor{yellow!20}}c|}{ First order CFO with $\beta =0$ }\\
%\multicolumn{3}{>{\columncolor{red}}l}{\color{white}\textsf{LONDON}}
\hline
h & $\|u_h-u \|_{0}$ & rate & $\|u_h-u\|_{1}$ & rate &$ \|q_h-q \|_{0}$& rate  & $\|\lambda_h\|_{0}$ & rate  \\
\hline
1/8    &    6.47e-04 &  -     &   1.64e-02 & \color{blue}{ -   }  &   2.94e-02 &\color{blue}{  -   }  &   5.98e-01  & -    \\
1/16   &    1.96e-04 &  1.72  &   8.18e-03 & \color{blue}{ 1.00}  &   1.30e-02 &\color{blue}{  1.18}  &   3.11e-01  & 0.95 \\
1/32   &    5.75e-05 &  1.77  &   4.08e-03 & \color{blue}{ 1.00}  &   6.07e-03 &\color{blue}{  1.09}  &   1.58e-01  & 0.98 \\
1/64   &    1.66e-05 &  1.79  &   2.04e-03 & \color{blue}{ 1.00}  &   2.95e-03 &\color{blue}{  1.04}  &   7.96e-02  & 0.99 \\
1/128  &    4.78e-06 &  1.80  &   1.02e-03 & \color{blue}{ 1.00}  &   1.46e-03 &\color{blue}{  1.01}  &   4.00e-02  & 0.99 \\
\hline
\multicolumn{9}{|>{\columncolor{yellow!20}}c|}{First order CFO with $\beta =-1$ }\\
\hline
\end{tabular}
\end{center}
\end{table}

\begin{table}[!ht]
\scriptsize
\begin{center}
\caption{Errors between the CFO scheme and the Ritz Galerkin finite element method for Test Case 3 with $\beta=1,2,0,-1$.}\label{table33-FE-beta120-1}
\begin{tabular}{|c|cc|cc|cc|cc|}
\hline
\multicolumn{5}{|>{\columncolor{yellow!20}}c|}{ First order CFO with $\beta =1$ }& \multicolumn{4}{|>{\columncolor{yellow!20}}c|}{ First order CFO with $\beta =0$ }\\
\hline
1/8    &   6.12e-04  &\color{blue}{ -    } &   1.74e-03 &\color{blue}{  -   }&         6.53e-04  & \color{blue}{-   }  &   1.87e-03 & \color{blue}{ -   }\\
1/16   &   1.93e-04  &\color{blue}{ 1.66 } &   7.01e-04 &\color{blue}{  1.31}&         2.13e-04  & \color{blue}{1.62}  &   7.81e-04 & \color{blue}{ 1.26}\\
1/32   &   5.59e-05  &\color{blue}{ 1.79 } &   2.45e-04 &\color{blue}{  1.52}&         6.33e-05  & \color{blue}{1.75}  &   2.82e-04 & \color{blue}{ 1.47}\\
1/64   &   1.56e-05  &\color{blue}{ 1.84 } &   8.21e-05 &\color{blue}{  1.57}&         1.82e-05  & \color{blue}{1.80}  &   9.92e-05 & \color{blue}{ 1.51}\\
1/128  &   4.25e-06  &\color{blue}{ 1.87 } &   2.76e-05 &\color{blue}{  1.57}&         5.18e-06  & \color{blue}{1.81}  &   3.56e-05 & \color{blue}{ 1.48}\\
h &  $\|u_h-R_hu \|_{0}$ & rate & $\|u_h- R_hu\|_{1}$ & rate &  $\|u_h-R_hu \|_{0}$ & rate & $\|u_h- R_hu\|_{1}$ & rate \\
\hline
\multicolumn{5}{|>{\columncolor{yellow!20}}c|}{ First order CFO with $\beta =2$ }& \multicolumn{4}{|>{\columncolor{yellow!20}}c|}{ First order CFO with $\beta =-1$ }\\
\hline
\end{tabular}
\end{center}
\end{table}
%case4%%%%%%%%%%%%%%%%%%%%%%%%%%%%%%%%%%%%%%%%%%%%%%%%%%%%%%%%%%%%%%%%%%%%%%%%%%%%%%%%%%%%%%%%%%%%%%%%%%%%%%%%%%
\begin{table}[!ht]
\scriptsize
\begin{center}
\caption{Error and convergence performance of the CFO scheme for Test Case 4 with $\beta=0,-1$.}\label{table44-h-beta-12}
\begin{tabular}{|c|cc|cc|cc|cc|}
\hline
\multicolumn{9}{|>{\columncolor{yellow!20}}c|}{ First order CFO with $\beta =0$ }\\
%\multicolumn{3}{>{\columncolor{red}}l}{\color{white}\textsf{LONDON}}
\hline
h & $\|u_h-u \|_{0}$ & rate & $\|u_h-u\|_{1}$ & rate &$ \|q_h-q \|_{0}$& rate  & $\|\lambda_h\|_{0}$ & rate  \\
\hline
1/8    &    3.36e-01  & -     &   2.41e-01 &\color{blue}{  -   }  &   3.69e-01 &\color{blue}{  -   }  &   4.22e+01  & -    \\
1/16   &    1.30e-01  & 1.38  &   1.12e-01 &\color{blue}{  1.10}  &   1.75e-01 &\color{blue}{  1.08}  &   3.09e+01  & 0.45 \\
1/32   &    4.16e-02  & 1.64  &   5.27e-02 &\color{blue}{  1.09}  &   7.68e-02 &\color{blue}{  1.19}  &   1.84e+01  & 0.75 \\
1/64   &    1.21e-02  & 1.78  &   2.54e-02 &\color{blue}{  1.05}  &   3.49e-02 &\color{blue}{  1.14}  &   9.92e+00  & 0.89 \\
1/128  &    3.31e-03  & 1.87  &   1.25e-02 &\color{blue}{  1.03}  &   1.66e-02 &\color{blue}{  1.07}  &   5.11e+00  & 0.96 \\
\hline
\multicolumn{9}{|>{\columncolor{yellow!20}}c|}{First order CFO with $\beta =-1$ }\\
\hline
\end{tabular}
\end{center}
\end{table}

\begin{table}[!ht]
\scriptsize
\begin{center}
\caption{Errors between the CFO scheme and the Ritz Galerkin finite element method for Test Case 4 with $\beta=1,2,0,-1$.}\label{table44-FE-beta120-1}
\begin{tabular}{|c|cc|cc|cc|cc|}
\hline
\multicolumn{5}{|>{\columncolor{yellow!20}}c|}{ First order CFO with $\beta =1$ }& \multicolumn{4}{|>{\columncolor{yellow!20}}c|}{ First order CFO with $\beta =0$ }\\
\hline
h &  $\|u_h-R_hu \|_{0}$ & rate & $\|u_h- R_hu\|_{1}$ & rate &  $\|u_h-R_hu \|_{0}$ & rate & $\|u_h- R_hu\|_{1}$ & rate \\
\hline
1/8    &    3.68e-02 &\color{blue}{  -    } &   2.02e-02  &\color{blue}{ -   }&         1.83e-01 & \color{blue}{ -   }  &   1.04e-01 & \color{blue}{ -   }\\
1/16   &    1.22e-02 &\color{blue}{  1.60 } &   5.95e-03  &\color{blue}{ 1.76}&         8.96e-02 & \color{blue}{ 1.03}  &   4.61e-02 & \color{blue}{ 1.18}\\
1/32   &    3.33e-03 &\color{blue}{  1.87 } &   1.63e-03  &\color{blue}{ 1.87}&         3.20e-02 & \color{blue}{ 1.49}  &   1.75e-02 & \color{blue}{ 1.40}\\
1/64   &    8.65e-04 &\color{blue}{  1.94 } &   4.46e-04  &\color{blue}{ 1.87}&         9.86e-03 & \color{blue}{ 1.70}  &   6.36e-03 & \color{blue}{ 1.46}\\
1/128  &    2.20e-04 &\color{blue}{  1.98 } &   1.18e-04  &\color{blue}{ 1.92}&         2.78e-03 & \color{blue}{ 1.83}  &   2.04e-03 & \color{blue}{ 1.64}\\
\hline
\multicolumn{5}{|>{\columncolor{yellow!20}}c|}{ First order CFO with $\beta =2$ }& \multicolumn{4}{|>{\columncolor{yellow!20}}c|}{ First order CFO with $\beta =-1$ }\\
\hline
\end{tabular}
\end{center}
\end{table}
\clearpage
\subsection{Lagrange multipliers}\label{Lagrangemultipliers}
~
%\begin{figure}[!h]
%\centering
%\subfigure[Smooth coefficient (left) and H\"older continuous coefficient (right)]{
%\label{Fig.sub.1.lam}
%\includegraphics [width=0.35\textwidth]{lambda_h1_test1_1order.png}
%\includegraphics [width=0.35\textwidth]{lambda_h1_test2_1order.png}
%}\\
%\subfigure[Discontinuous coefficients: test case 3 (left) and test case 4 (right)]{
%\label{Fig.sub.2.lam}
%\includegraphics [width=0.35\textwidth]{lambda_h1_test3_1order.png}
%\includegraphics [width=0.35\textwidth]{lambda_h1_test4_1order.png}
%}\\
%\caption{The solution profile for the Lagrange multiplier $\lambda_h$ on a partition of size $64\times 64$
% arising from the first order CFO scheme (\ref{min-LagrangeForm2-1}-\ref{min-LagrangeForm2-2}).}
%\label{fig:lambda}
%\end{figure}
\begin{figure}[!h]
\centering
\subfigure[Smooth coefficient (left) and H\"older continuous coefficient (right)]{
\label{Fig.sub.1.lam2order}
\includegraphics [width=0.35\textwidth]{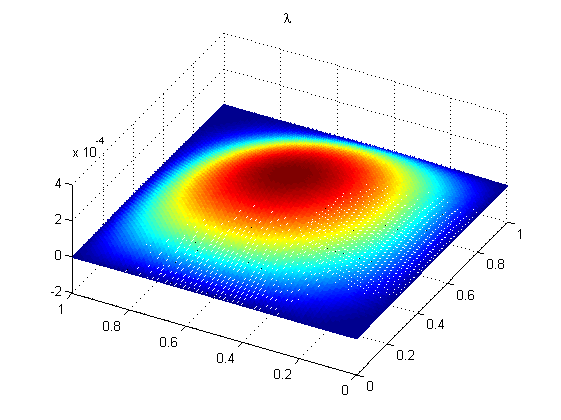}
\includegraphics [width=0.35\textwidth]{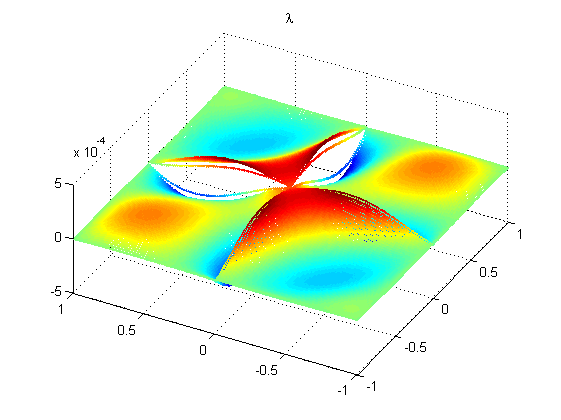}
}\\
\subfigure[Discontinuous coefficients: test case 3 (left) and test case 4 (right)]{
\label{Fig.sub.2.lam2order}
\includegraphics [width=0.35\textwidth]{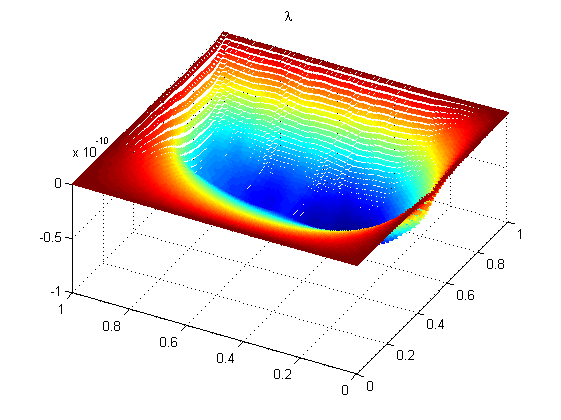}
\includegraphics [width=0.35\textwidth]{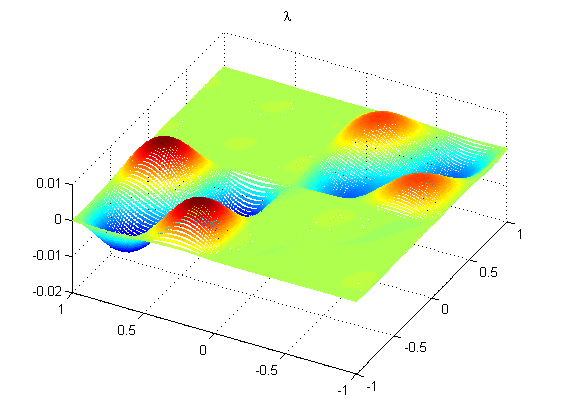}
}\\
\caption{The solution profile for the Lagrange multiplier $\lambda_h$ on a partition of size $64\times 64$
 arising from the second order CFO scheme (\ref{min-LagrangeForm2-1}-\ref{min-LagrangeForm2-2}).}
\label{fig:lambda2order}
\end{figure}

\begin{figure}[!ht]
\centering
\subfigure[Smooth coefficient (left) and H\"older continuous coefficient (right)]{
\label{Fig.sub.1.lam3order}
\includegraphics [width=0.35\textwidth]{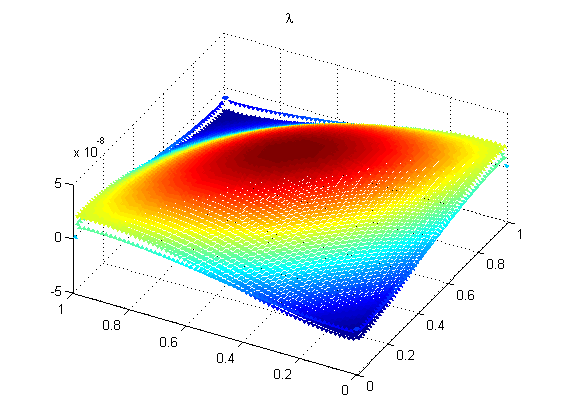}
\includegraphics [width=0.35\textwidth]{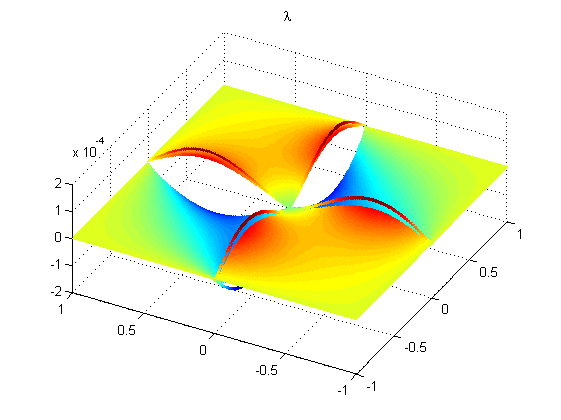}
}\\
\subfigure[Discontinuous coefficients: test case 3 (left) and test case 4 (right)]{
\label{Fig.sub.2.lam3order}
\includegraphics [width=0.35\textwidth]{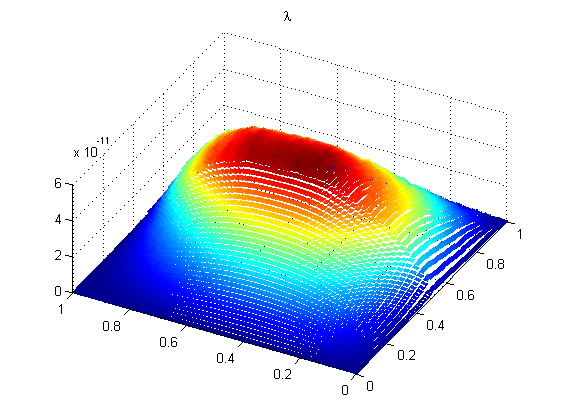}
\includegraphics [width=0.35\textwidth]{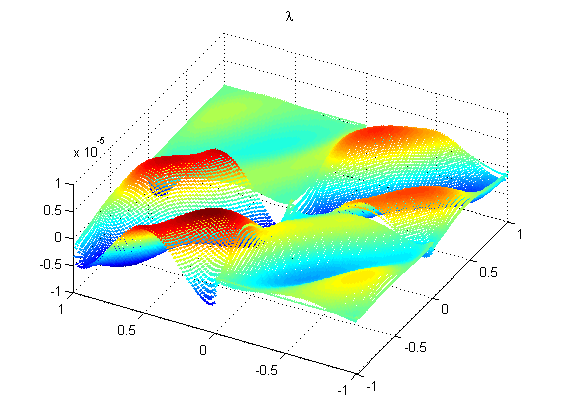}
}\\
\caption{The solution profile for the Lagrange multiplier $\lambda_h$ on a partition of size $64\times 64$
 arising from the third order CFO scheme (\ref{min-LagrangeForm2-1}-\ref{min-LagrangeForm2-2}).}
\label{fig:lambda3order}
\end{figure}
\clearpage
%%LMtestcase2%%%%%%%%%%%%%%%%%%%%%%%%%%%%%%%%%%%%%%%%%%%%%%%%%%%%%%%%%%%%%%%%%%%%%%%%%%%%%%%%%%%%%%%%%%%%%%%%%%%%%%%%%%%%%%%%%%%%%%%%%%%%%%%%%%%%%
\begin{figure}
\centering
\subfigure[$\lambda_h^2$ (left, stereogram above, plane diagram below) and $\|u_h-u \|^2_{0}$ (right, stereogram above, plane diagram below) ]{
\label{Fig.sub.1.lam1order-error-test2}
\includegraphics [width=0.95\textwidth]{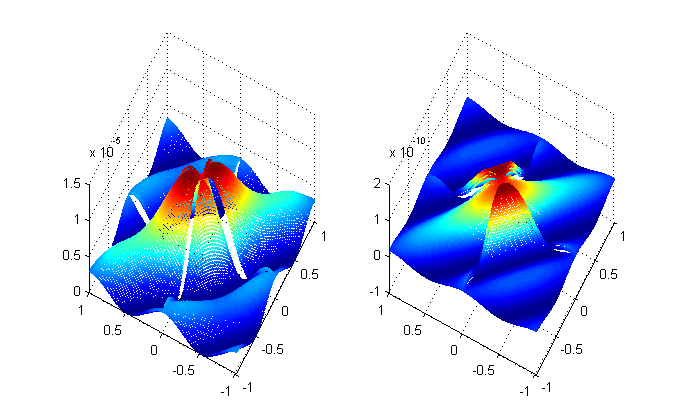}
}\\
\includegraphics [width=0.95\textwidth]{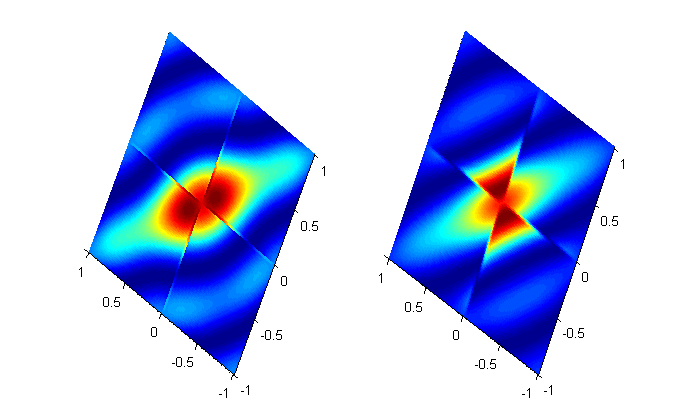}
\caption{The solution profile for the square of Lagrange multiplier $\lambda_h^2$ on a partition of size $64\times 64$
 arising from the first order CFO scheme (\ref{min-LagrangeForm2-1}-\ref{min-LagrangeForm2-2}) and the $L_2$ error of the primal variable $\|u_h-u \|^2_{0}$ for test case 2.}
\label{fig:lambda1order-error-l2-test2}
\end{figure}

%%LMtestcase3%%%%%%%%%%%%%%%%%%%%%%%%%%%%%%%%%%%%%%%%%%%%%%%%%%%%%%%%%%%%%%%%%%%%%%%%%%%%%%%%%%%%%%%%%%%%%%%%%%%%%%%%%%%%%%%%%%%%%%%%%%%%%%%%%%%%%
\begin{figure}
\centering
\subfigure[$\lambda_h^2$ (left, stereogram above, plane diagram below) and $\|u_h-u \|^2_{0}$ (right, stereogram above, plane diagram below) ]{
\label{Fig.sub.1.lam1order-error-test3}
\includegraphics [width=0.95\textwidth]{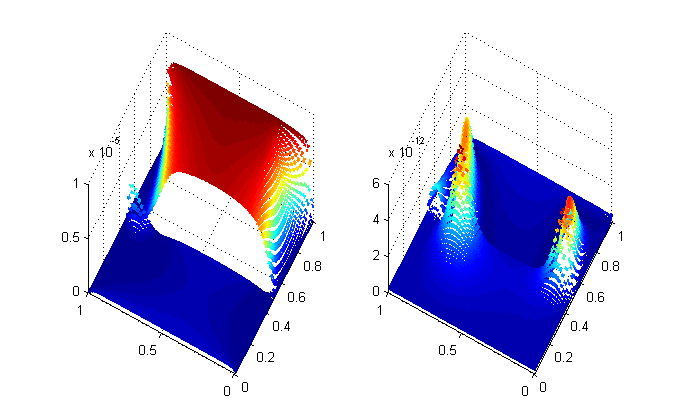}
}\\
\includegraphics [width=0.95\textwidth]{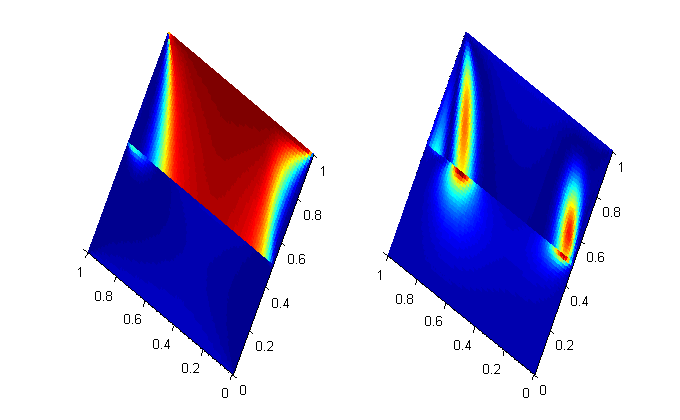}
\caption{The solution profile for the square of Lagrange multiplier $\lambda_h^2$ on a partition of size $64\times 64$
 arising from the first order CFO scheme (\ref{min-LagrangeForm2-1}-\ref{min-LagrangeForm2-2}) and the $L_2$ error of the primal variable $\|u_h-u \|^2_{0}$ for test case 3.}
\label{fig:lambda1order-error-l2-test3}
\end{figure}

%%LMtestcase4%%%%%%%%%%%%%%%%%%%%%%%%%%%%%%%%%%%%%%%%%%%%%%%%%%%%%%%%%%%%%%%%%%%%%%%%%%%%%%%%%%%%%%%%%%%%%%%%%%%%%%%%%%%%%%%%%%%%%%%%%%%%%%%%%%%%%
\begin{figure}
\centering
\subfigure[$\lambda_h^2$ (left, stereogram above, plane diagram below) and $\|u_h-u \|^2_{0}$ (right, stereogram above, plane diagram below) ]{
\label{Fig.sub.1.lam1order-error-test4}
\includegraphics [width=0.95\textwidth]{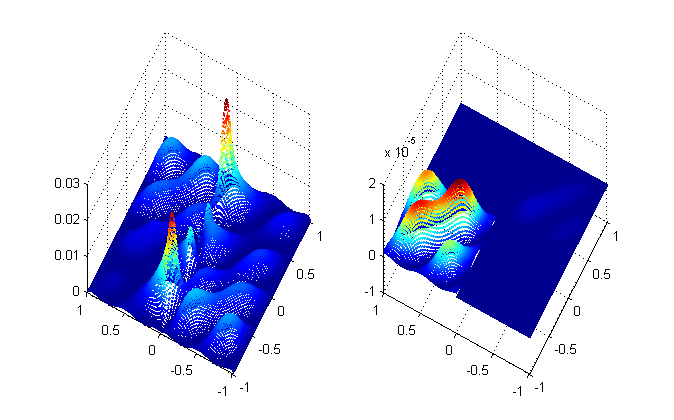}
}\\
\includegraphics [width=0.95\textwidth]{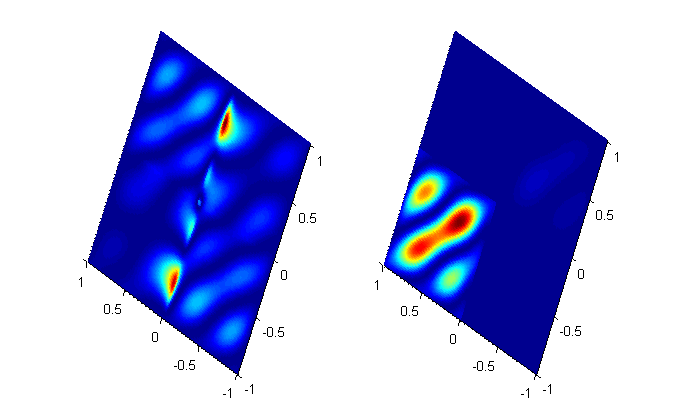}
\caption{The solution profile for the square of Lagrange multiplier $\lambda_h^2$ on a partition of size $64\times 64$
 arising from the first order CFO scheme (\ref{min-LagrangeForm2-1}-\ref{min-LagrangeForm2-2}) and the $L_2$ error of the primal variable $\|u_h-u \|^2_{0}$ for test case 4.}
\label{fig:lambda1order-error-l2-test4}
\end{figure}
%%LMtestcase1 2 order %%%%%%%%%%%%%%%%%%%%%%%%%%%%%%%%%%%%%%%%%%%%%%%%%%%%%%%%%%%%%%%%%%%%%%%%%%%%%%%%%%%%%%%%%%%%%%%%%%%%%%%%%%%%%%%%%%%%%
\begin{figure}
\centering
\subfigure[$\lambda_h^2$ (left, stereogram above, plane diagram below) and $\|u_h-u \|^2_{0}$ (right, stereogram above, plane diagram below) ]{
\label{Fig.sub.1.lam2order-error}
\includegraphics [width=0.95\textwidth]{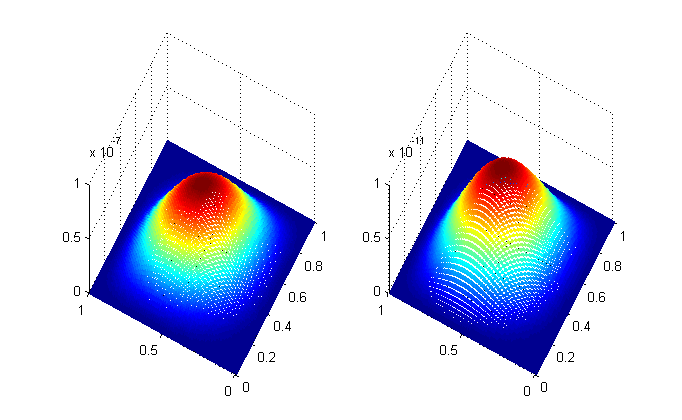}
}\\
\includegraphics [width=0.95\textwidth]{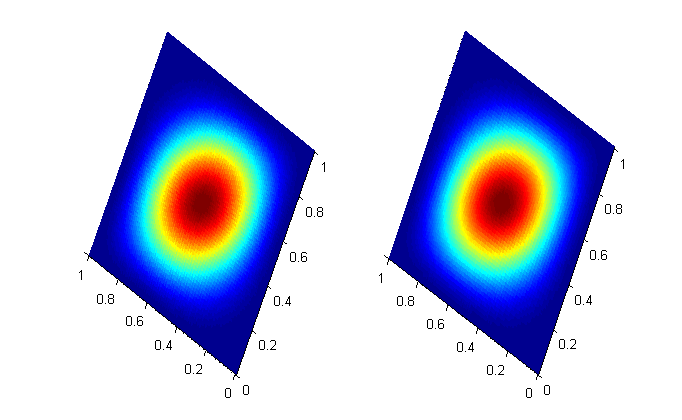}
\caption{The solution profile for the square of Lagrange multiplier $\lambda_h^2$ on a partition of size $64\times 64$
 arising from the second order CFO scheme (\ref{min-LagrangeForm2-1}-\ref{min-LagrangeForm2-2}) and the $L_2$ error of the primal variable $\|u_h-u \|^2_{0}$ for test case 1.}
\label{fig:lambda2order-error-l2}
\end{figure}

%%LMtestcase2%%%%%%%%%%%%%%%%%%%%%%%%%%%%%%%%%%%%%%%%%%%%%%%%%%%%%%%%%%%%%%%%%%%%%%%%%%%%%%%%%%%%%%%%%%%%%%%%%%%%%%%%%%%%%%%%%%%%%%%%%%%%%%%%%%%%%
\begin{figure}
\centering
\subfigure[$\lambda_h^2$ (left, stereogram above, plane diagram below) and $\|u_h-u \|^2_{0}$ (right, stereogram above, plane diagram below) ]{
\label{Fig.sub.1.lam2order-error-test2}
\includegraphics [width=0.95\textwidth]{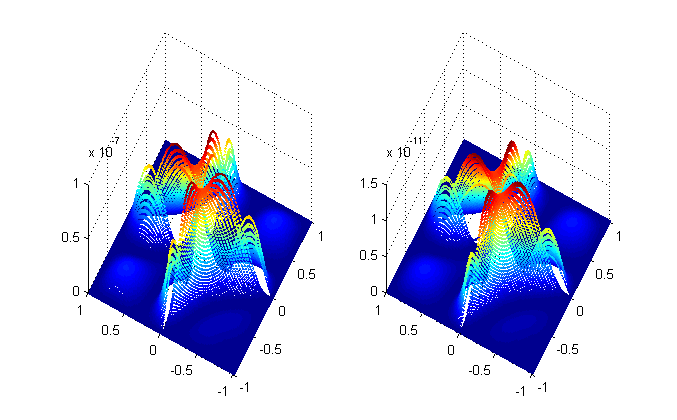}
}\\
\includegraphics [width=0.95\textwidth]{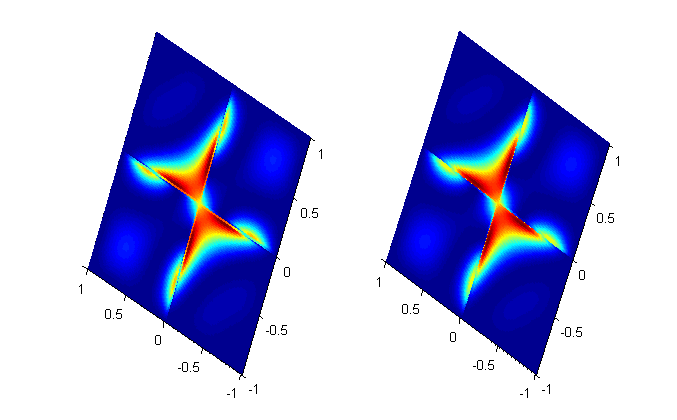}
\caption{The solution profile for the square of Lagrange multiplier $\lambda_h^2$ on a partition of size $64\times 64$
 arising from the second order CFO scheme (\ref{min-LagrangeForm2-1}-\ref{min-LagrangeForm2-2}) and the $L_2$ error of the primal variable $\|u_h-u \|^2_{0}$ for test case 2.}
\label{fig:lambda2order-error-l2-test2}
\end{figure}

%%LMtestcase3%%%%%%%%%%%%%%%%%%%%%%%%%%%%%%%%%%%%%%%%%%%%%%%%%%%%%%%%%%%%%%%%%%%%%%%%%%%%%%%%%%%%%%%%%%%%%%%%%%%%%%%%%%%%%%%%%%%%%%%%%%%%%%%%%%%%%
\begin{figure}
\centering
\subfigure[$\lambda_h^2$ (left, stereogram above, plane diagram below) and $\|u_h-u \|^2_{0}$ (right, stereogram above, plane diagram below) ]{
\label{Fig.sub.1.lam2order-error-test3}
\includegraphics [width=0.95\textwidth]{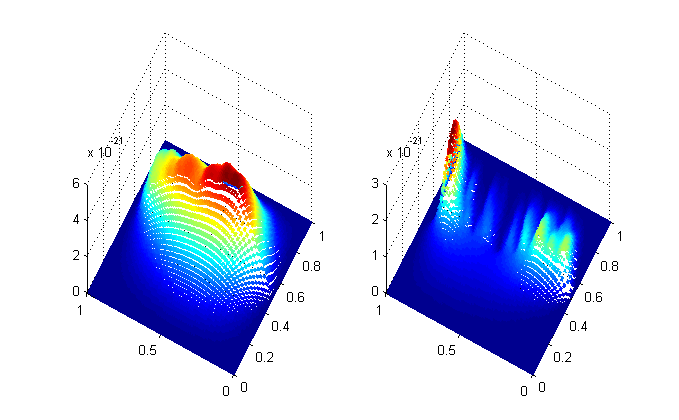}
}\\
\includegraphics [width=0.95\textwidth]{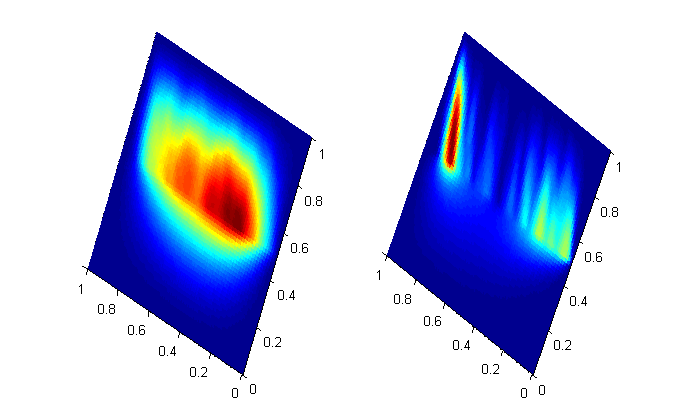}
\caption{The solution profile for the square of Lagrange multiplier $\lambda_h^2$ on a partition of size $64\times 64$
 arising from the second order CFO scheme (\ref{min-LagrangeForm2-1}-\ref{min-LagrangeForm2-2}) and the $L_2$ error of the primal variable $\|u_h-u \|^2_{0}$ for test case 3. (Since the exact solution of test case 3 is second order polyoma, $\|u_h-u \|^2_{0}$) is of machine precision, we multiple $\|u_h-u \|^2_{0}$ by $10^5$ to show it.)}
\label{fig:lambda2order-error-l2-test3}
\end{figure}

%%LMtestcase4%%%%%%%%%%%%%%%%%%%%%%%%%%%%%%%%%%%%%%%%%%%%%%%%%%%%%%%%%%%%%%%%%%%%%%%%%%%%%%%%%%%%%%%%%%%%%%%%%%%%%%%%%%%%%%%%%%%%%%%%%%%%%%%%%%%%%
\begin{figure}
\centering
\subfigure[$\lambda_h^2$ (left, stereogram above, plane diagram below) and $\|u_h-u \|^2_{0}$ (right, stereogram above, plane diagram below) ]{
\label{Fig.sub.1.lam2order-error-test4}
\includegraphics [width=0.95\textwidth]{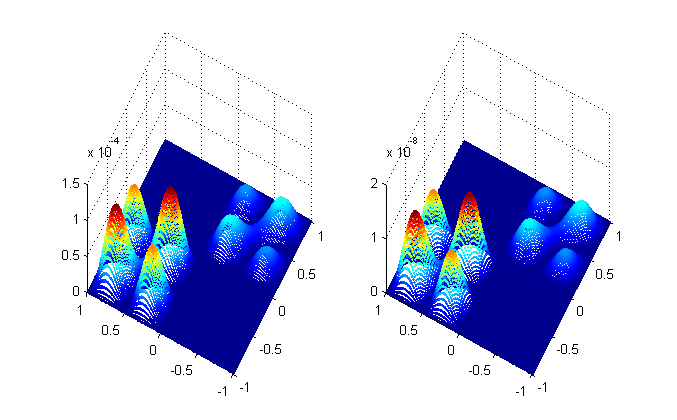}
}\\
\includegraphics [width=0.95\textwidth]{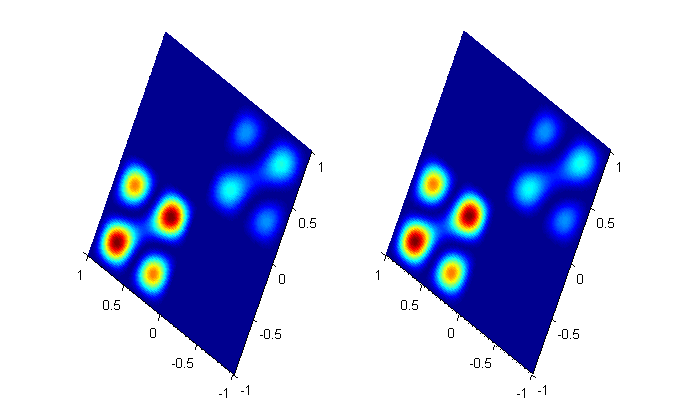}
\caption{The solution profile for the square of Lagrange multiplier $\lambda_h^2$ on a partition of size $64\times 64$
 arising from the second order CFO scheme (\ref{min-LagrangeForm2-1}-\ref{min-LagrangeForm2-2}) and the $L_2$ error of the primal variable $\|u_h-u \|^2_{0}$ for test case 4.}
\label{fig:lambda2order-error-l2-test4}
\end{figure}

%%LMtestcase1 3 order%%%%%%%%%%%%%%%%%%%%%%%%%%%%%%%%%%%%%%%%%%%%%%%%%%%%%%%%%%%%%%%%%%%%%%%%%%%%%%%%%%%%%%%%%%%%%%%%%%%%%%%%%%%%%%%%
\begin{figure}
\centering
\subfigure[$\lambda_h^2$ (left, stereogram above, plane diagram below) and $\|u_h-u \|^2_{0}$ (right, stereogram above, plane diagram below) ]{
\label{Fig.sub.1.lam3order-error}
\includegraphics [width=0.95\textwidth]{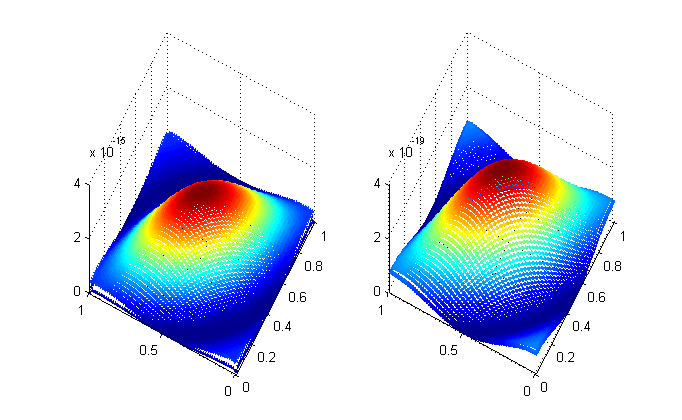}
}\\
\includegraphics [width=0.95\textwidth]{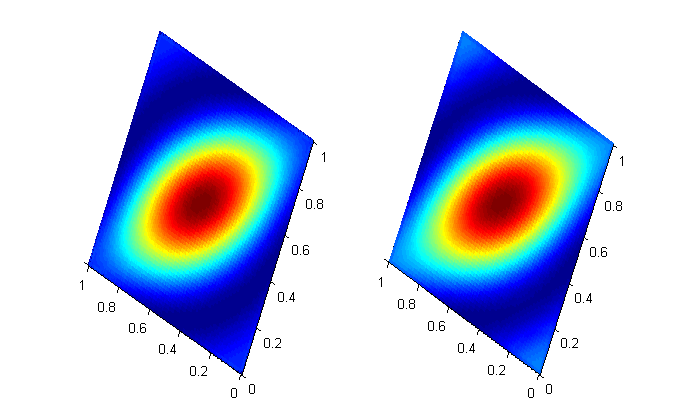}
\caption{The solution profile for the square of Lagrange multiplier $\lambda_h^2$ on a partition of size $64\times 64$
 arising from the third order CFO scheme (\ref{min-LagrangeForm2-1}-\ref{min-LagrangeForm2-2}) and the $L_2$ error of the primal variable $\|u_h-u \|^2_{0}$ for test case 1.}
\label{fig:lambda3order-error-l2}
\end{figure}

%%LMtestcase2%%%%%%%%%%%%%%%%%%%%%%%%%%%%%%%%%%%%%%%%%%%%%%%%%%%%%%%%%%%%%%%%%%%%%%%%%%%%%%%%%%%%%%%%%%%%%%%%%%%%%%%%%%%%%%%%%%%%%%%%%%%%%%%%%%%%%
\begin{figure}
\centering
\subfigure[$\lambda_h^2$ (left, stereogram above, plane diagram below) and $\|u_h-u \|^2_{0}$ (right, stereogram above, plane diagram below) ]{
\label{Fig.sub.1.lam3order-error-test2}
\includegraphics [width=0.95\textwidth]{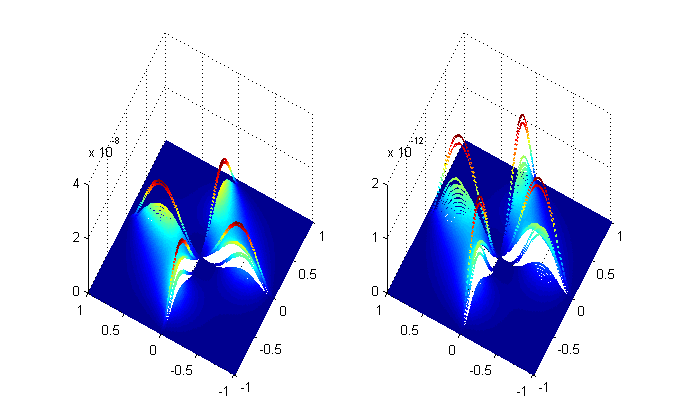}
}\\
\includegraphics [width=0.95\textwidth]{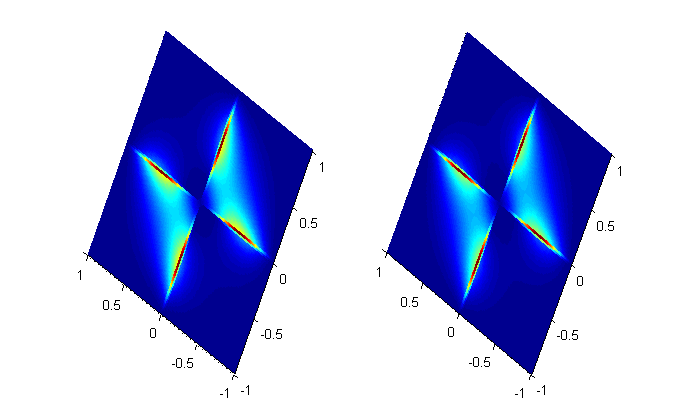}
\caption{The solution profile for the square of Lagrange multiplier $\lambda_h^2$ on a partition of size $64\times 64$
 arising from the third order CFO scheme (\ref{min-LagrangeForm2-1}-\ref{min-LagrangeForm2-2}) and the $L_2$ error of the primal variable $\|u_h-u \|^2_{0}$ for test case 2.}
\label{fig:lambda3order-error-l2-test2}
\end{figure}

%%LMtestcase3%%%%%%%%%%%%%%%%%%%%%%%%%%%%%%%%%%%%%%%%%%%%%%%%%%%%%%%%%%%%%%%%%%%%%%%%%%%%%%%%%%%%%%%%%%%%%%%%%%%%%%%%%%%%%%%%%%%%%%%%%%%%%%%%%%%%%
\begin{figure}
\centering
\subfigure[$\lambda_h^2$ (left, stereogram above, plane diagram below) and $\|u_h-u \|^2_{0}$ (right, stereogram above, plane diagram below) ]{
\label{Fig.sub.1.lam3order-error-test3}
\includegraphics [width=0.95\textwidth]{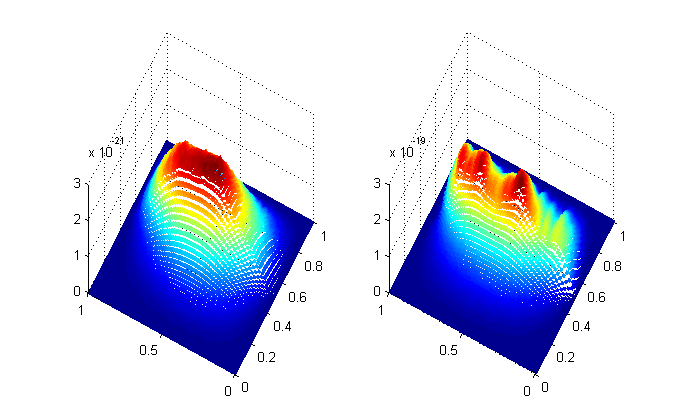}
}\\
\includegraphics [width=0.95\textwidth]{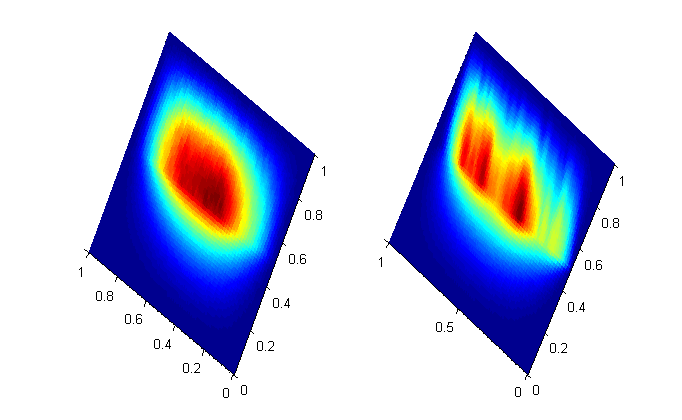}
\caption{The solution profile for the square of Lagrange multiplier $\lambda_h^2$ on a partition of size $64\times 64$
 arising from the third order CFO scheme (\ref{min-LagrangeForm2-1}-\ref{min-LagrangeForm2-2}) and the $L_2$ error of the primal variable $\|u_h-u \|^2_{0}$ for test case 3. (Since the exact solution is second order polyoma, $\|u_h-u \|^2_{0}$) is of machine precision, we multiple $\|u_h-u \|^2_{0}$ by $10^5$ to show it.)}
\label{fig:lambda3order-error-l2-test3}
\end{figure}

%%LMtestcase4%%%%%%%%%%%%%%%%%%%%%%%%%%%%%%%%%%%%%%%%%%%%%%%%%%%%%%%%%%%%%%%%%%%%%%%%%%%%%%%%%%%%%%%%%%%%%%%%%%%%%%%%%%%%%%%%%%%%%%%%%%%%%%%%%%%%%
\begin{figure}
\centering
\subfigure[$\lambda_h^2$ (left, stereogram above, plane diagram below) and $\|u_h-u \|^2_{0}$ (right, stereogram above, plane diagram below) ]{
\label{Fig.sub.1.lam3order-error-test4}
\includegraphics [width=0.95\textwidth]{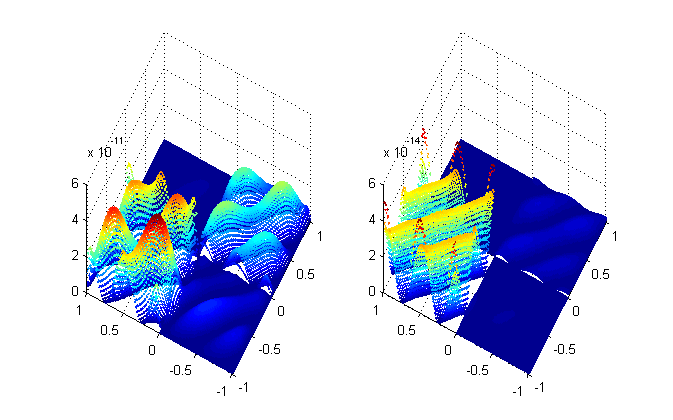}
}\\
\includegraphics [width=0.95\textwidth]{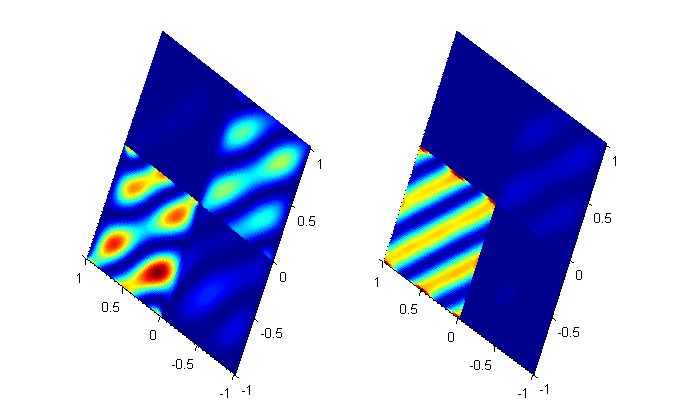}
\caption{The solution profile for the square of Lagrange multiplier $\lambda_h^2$ on a partition of size $64\times 64$
 arising from the third order CFO scheme (\ref{min-LagrangeForm2-1}-\ref{min-LagrangeForm2-2}) and the $L_2$ error of the primal variable $\|u_h-u \|^2_{0}$ for test case 4.}
\label{fig:lambda3order-error-l2-test4}
\end{figure}


\begin{thebibliography}{99}

\bibitem{Arnold_SIAMJNA_2002}
D. N. Arnold, F. Brezzi, B. Cockburn, L. D. Marini.
\newblock Unified analysis of discontinuous Galerkin methods for elliptic problems.
 {\it SIAM journal on numerical analysis.}  39(5): 1749-1779, 2002.

\bibitem{Aziz_Chapman_1979}
K. Aziz and A. Settari.
\newblock  Petroleum reservoir simulation.
 {\it  Chapman \& Hall.} 1979.
%
%\bibitem{Babuska1973} I. Babu$\breve{s}$ka.
%\newblock{ The finie element method with Lagrange multipliers},
%\newblock {\em Numer. Math.}, 20 : 173-192, 1973.

\bibitem{Bank.R;Rose.D1987}
R. E. Bank and D. J. Rose.
\newblock Some error estimates for the box scheme.
\newblock {\em SIAM  Numer. Anal.}, 24:777--787, 1987.

\bibitem{Barth.T;Ohlberger2004}
T.~Barth and M.~Ohlberger.
\newblock Finite volume methods: foundation and analysis.
\newblock In {\em Encyclopedia of computational Mechanics, volume~1,
  chapter~15}. John Wiley \& Sons, 2004.

%\bibitem{Baumann_CMAME_1999}
%C. E. Baumann and J. T. Oden.
%\newblock  A discontinuous hp finite element method for convection-diffusion problems.
% {\it Computer Methods in Applied Mechanics and Engineering.} 175(3-4): 311-341, 1999.
%
%\bibitem{Bochev}
%P. B. Bochev and M. D. Gunzburger.
%\newblock A locally conservative least-squares method for Darcy flows.
%{\em Commun. Numer. Meth. Engng} 2006, 00:1每6.
%
%\bibitem{Brezzi1974}F. Brezzi.
%\newblock {On the existence, uniqueness, and approximation of saddle point problems arising from Lagrange multipliers,}
%\newblock{\em RAIRO}, 8: 129-151, 1974.

\bibitem{bdm} F. Brezzi, J. Douglas, and L. Marini.
\newblock Two families of mixed finite elements for second order elliptic problems.
{\em Numer. Math.} 47 (1985) 217-235.

\bibitem{Bush_SIAMJSC_2013}
L. Bush and V. Ginting.
\newblock  On the application of the continuous Galerkin finite element method for conservation problems.
 {\it SIAM Journal on Scientific Computing.} 35(6): A2953-A2975, 2013.
%
\bibitem{Bush_JCAM_2014}
L. Bush, V. Ginting and  M. Presho.
\newblock Application of a conservative, generalized multiscale finite element method to flow models.
 {\it  Journal of Computational and Applied Mathematics.} 260: 395-409, 2014.

\bibitem{Cai.Z1991}
Z. Cai.
\newblock On the finite volume element method.
\newblock {\em Numer. Math.}, 58:713--735, 1991.

\bibitem{CaiDouglasPark2003} Z. Cai, J. Douglas, and M. Park.
\newblock Development and analysis of higher order finite volume methods over rectangles for elliptic equations.
\newblock {\em Challenges in computational mathematics (Pohang, 2001). Adv. Comput. Math.} 19 (2003), no. 1-3, 3每33.

\bibitem{Carstensen_MCAMS_1997}
\text{C. Carstensen.}
\newblock A posteriori error estimate for the mixed finite element method[J].
\newblock {\em Mathematics of Computation of the American Mathematical Society}, 1997, 66(218): 465-476.

\bibitem{Chen2010} L. Chen.
\newblock A new class of high order finite volume methods for second order elliptic equations.
\newblock {\em SIAM J. Numer. Anal.} 47 (2010), no. 6, 4021每4043.

\bibitem{ChenWuXu2012} Z. Chen, J. Wu, and Y. Xu,
\newblock Higher-order finite volume methods for elliptic boundary value problems.
\newblock {\em Adv. Comput. Math.}  37(2012), 191-253.

\bibitem{ChouLi2000} S. ~H. Chou and Q. Li.
 Error Estimates in $L^2 ,H^1$ and $L^\infty $ in covolume methods for elliptic and parabolic problems: a unified approach mathematics of computation,
{\it Math. Comp.}  69: 103-120, 2000.

\bibitem{Christie_SPE_2001}
M.  A. Christie and M. J. Blunt.
\newblock Tenth SPE comparative solution project: A comparison of upscaling techniques.
{\it SPE reservoir simulation symposium.} Society of Petroleum Engineers, 2001.

%\bibitem{s-chou}
%S. H. Chou and S. Tang.
%Conservative $P1$ conforming and nonconforming Galerkin FEMS: effective flux evaluation via a nonmixed method approach, {\it SIAM J. Numer. Anal.}, Vol 38, No. 2, pp. 660-680, 2000.
%
%\bibitem{Cockburn_Jay} B. Cockburn, J. Gopalakrishnan, and H. Wang.
%\newblock Locally conservative fluxes for the continuous Galerkin method.
%{\em SIAM J. Numer. Anal.}, 45(4)(2007), 1742每1776.
%
%\bibitem{Cockburn_SIAMJNA_1998}
%B. Cockburn and  C. W. Shu.
%\newblock  The local discontinuous Galerkin method for time-dependent convection-diffusion systems.
% {\it  SIAM Journal on Numerical Analysis.} 35(6): 2440-2463, 1998.

\bibitem{EfendievGinting_JCP_2006}
Y. Efendiev, V. Ginting, T. Hou, and R. Ewing.
\newblock Accurate multiscale finite element methods for two-phase flow simulations.
{\it Journal of Computational Physics.} 220(1), 155-174, 2006.

\bibitem{Emonot1992}
Ph. Emonot. Methods de volums elements finis: applications aux equations de
  navier-stokes et resultats de convergence.
\newblock {\em Lyon}, 1992.


\bibitem{EymardGallouetHerbin2000}
R. Eymard, T. Gallouet, and R. Herbin.
\newblock {\em Finite Volume Methods}, In : {\em Handbook of Numerical Analysis}, VII,  713-1020,  P. G. Ciarlet and J. L. Lions Eds.,
\newblock North-Holland, Amsterdam, 2000.

%\bibitem{EymardGallouet2002}
%R. Eymard, T. Gallou\"et, R. Herbin.
%\newblock A cell-centred finite-volume approximation for anisotropic diffusion operators on unstructured meshes in any space dimension.
%\newblock {\em IMA Journal of Numerical Analysis}, 26(2): 326-353, 2006.
%
%\bibitem{EymardGallouet2006}
%R. Eymard, T. Gallou\"et, R. Herbin.
%\newblock A cell-centred finite-volume approximation for anisotropic diffusion operators on unstructured meshes in any space dimension.
%\newblock {\em IMA Journal of Numerical Analysis}, 26(2): 326-353, 2006.

\bibitem{Gilbarg_NM_1983}
 D.~Gilbarg and N.S.~Trudinger.
\newblock Elliptic partial differential equations of second order.
\newblock In {\em  Springer-Verlag,} Berlin-Heidelberg-New York, 1983.

\bibitem{Hackbusch.W1989a}
W. Hackbusch.
\newblock On first and second order box methods.
\newblock {\em Computing}, 41:277--296, 1989.

%\bibitem{T-Hughes}
%T. J. R. Hughes, G. Engel, L. Mazzei, and M. G. Larson.
%\newblock The continuous Galerkin method is locally conservative.
%{\em Journal of Computational Physics} 163, 467每488 (2000).
%
%\bibitem{Jeon _SIAMJNA_2010}
%Y. Jeon and E. J. Park.
%\newblock  A hybrid discontinuous Galerkin method for elliptic problems.
%  {\it SIAM Journal on Numerical Analysis.}  48(5): 1968-1983, 2010.
%
%\bibitem{Larson}
%M. G. Larson and A. J. Niklasson.
%\newblock A conservative flux for the continuous Galerkin method based on discontinuous enrichment.
%{\em CALCOLO} 41, 65-76 (2004).
%
%\bibitem{Lazarov1996} R. Lazarov, I. Michev, and P. Vassilevski.
%\newblock{Finite volume methods for convection-diffusion
%problems}.
%\newblock{\em SIAM J. Numer. Anal.}, 33 :31-55, 1996.

\bibitem{LeVeque2002}
R. J. LeVeque.
\newblock Finite volume methods for hyperbolic problems.
\newblock {\em Cambridge university press}, 2002.

\bibitem{Li.R2000} R. Li, Z. Chen, and W. Wu.
\newblock {\em The Generalized Difference Methods for Partial differential Equations}.
\newblock Marcel Dikker, New York, 2000.

\bibitem{LinYangZou2015} Y. Lin, M. Yang and Q. Zou.
\newblock  $L^2$ error estimates for a class of any order finite volume schemes over quadrilateral meshes.
\newblock {\em SIAM J. Numer. Anal.}, 53:2009--2029, 2015.

\bibitem{LiuWang_SINUM_2017}
\text{Y. Liu, J. Wang and Q. Zou.}
\newblock  A Conservative Flux Optimization Finite Element Method for Convection-Diffusion Equations.
\newblock {\em SIAM Journal on Numerical Analysis}, 2019, 57(3): 1238-1262.

%\bibitem{Loula} A.F.D. Loula, F.A. Rochinha, and M.A. Murad.
%\newblock Higher-order gradient post-processings for second-order elliptic
%problems.
%{\em Comput. Methods Appl. Mech. Engrg.} 128 (1995) 361-381.
%
%\bibitem{Mudunuro}
%M. K. Mudunuru and K. B. Nakshatrala.
%\newblock On enforcing maximum principles and achieving element-wise species balance for advection-diffusion-reaction equations under the finite element method.
%{\em Journal of Computational Physics}, Volume 305, 15 (2016), pp. 448-493.

\bibitem{Nguyen_JCP_2009}
N. C. Nguyen, J. Peraire, and B. Cockburn.
\newblock  An implicit high-order hybridizable discontinuous Galerkin method for linear convection-diffusion equations.
 {\it Journal of Computational Physics.} 228(9): 3232-3254, 2009.

\bibitem{Nicolaides1995}  R. A. Nicolaides, T. A. Porsching, and C. A. Hall.
\newblock Covolume methods in computational fluid dynamics,
\newblock In {\em  Computational Fluid Dynamics Review,} M. Hafez and K. Oshima, eds.
Wiley, New York, 1995, 279--299.

%\bibitem{Odsater} L. H. Odsater, M. F. Wheeler, T. Kvamsdala, and M. G. Larson.
%\newblock Postprocessing of non-conservative flux for compatibility with transport in heterogeneous media.
%{\em Computer Methods in Applied Mechanics and Engineering},
%315 (2017) 799-830.
%
%\bibitem{RanacherScott1982} R. Rannacher  and R. Scott. Some optimal error estimate for piecewise linear finite element approximations,
%{\em Math. Comp.}, 38: 437-445, 1982.

\bibitem{RaviartThomas} P. A. Raviart and J. M. Thomas. A mixed finite element method for second order elliptic problems, In {\it Mathematical Aspects of Finite Element Methods}, I. Galligani and E. Magenes, eds. Springer-Verlag, Berlin-Heidelhberg-New York, 1977, 292-315.

%\bibitem{Schmidt1993}
%T. Schmidt.
%\newblock{Box schemes on quadrilateral meshes.}
%\newblock{\em Computing}, 51:271-292, 1993.
%

\bibitem{Shu2003}C. Shu. High order finite difference and finite
volume WENO schemes and discontinuous Galerkin methods for CFD.
{\it Journal of Computational Fluid Dynamics}, {\bf
17}(2003),107-118.

\bibitem{Smears_SIAMJNA_2013}
I. Smears and E.  S\"uli.
\newblock Discontinuous Galerkin finite element approximation of nondivergence form elliptic equations with Cordes coefficients.
 {\it SIAM Journal on Numerical Analysis.} 51(4): 2088-2106, 2013.

\bibitem{Suli_MCOM_1992}
\textsc{S\"{u}li,~E.}
The accuracy of cell vertex finite volume methods on quadrilateral meshes.
{\it Math. Comput.}, {59}, 359--382,1992.

\bibitem{WangWang_2016}
C. Wang and J. Wang.
\newblock{A primal-dual weak Galerkin finite element method for second elliptic equations in non-divergence form},
{\it Math. Comp.}, DOI: https://doi.org/10.1090/mcom/3220. June 2017.

\bibitem{WangYe_2013}
J. Wang and X. Ye.
\newblock{A weak Galerkin mixed finite element method for second-order ellliptic problems.}
available at arXiv: 1104.2897vl.
{\it J. Comp. and Appl. Math.}, {\bf 241}, 103-115, 2013.

\bibitem{wy3655}
{\sc J. Wang and X. Ye}. {\em A weak Galerkin mixed finite element
method for second-order elliptic problems}, Math. Comp., 83 (2014), pp. 2101-2126.

\bibitem{Yubo2012}
G. Yu, B. Yu, Y. Zhao, et al.
\newblock Comparative studies on accuracy and convergence rate between the cell-centered scheme and the cell-vertex scheme for triangular grids.
\newblock {\em International Journal of Heat and Mass Transfer}, 55(25): 8051-8060,2012.

%\bibitem{RaviartThomas_MFE_1977}
%\text{P. A. Raviart, J. M. Thomas.}
%\newblock A mixed finite element method for 2-nd order elliptic problems.
%\newblock {\em Mathematical aspects of finite element methods}, Springer, Berlin, Heidelberg, 1977: 292-315.

\end{thebibliography}
\end{document}